\documentclass[12pt,reqno]{amsart}
\usepackage{geometry,graphicx,enumerate} 
\usepackage{dsfont}

\numberwithin{equation}{section}

\theoremstyle{plain}
\newtheorem{theorem}{Theorem}[section]
\newtheorem{lemma}[theorem]{Lemma}
\newtheorem{proposition}[theorem]{Proposition}

\newtheorem{corollary}[theorem]{Corollary}

\theoremstyle{definition}
\newtheorem{definition}{Definition}[section]

\theoremstyle{remark}
\newtheorem{remark}{Remark}[section]
\newcommand {\calN}        {{\mathcal N}}
\newcommand {\calF}        {{\mathcal F}}
\newcommand {\calS}        {{\mathcal S}}
\newcommand{\bitem}{\begin{itemize}}
\newcommand{\eitem}{\end{itemize}}
\newcommand{\mc}[1]{\mathcal{#1}}

\newcommand {\Rmn}        {{\mathbb R^{m\times n}}}

\newcommand{\N}{\mathbb{N}}
\newcommand{\R}{\mathbb{R}}

\newcommand{\EE}{\mathbb{E}}

\newcommand{\bpm}{\begin{pmatrix}}
\newcommand{\epm}{\end{pmatrix}}
\newcommand{\bsm}{\left(\begin{smallmatrix}}
\newcommand{\esm}{\end{smallmatrix}\right)}
\newcommand{\T}{\top}

\newcommand{\mrm}[1]{\mathrm{#1}}

\newcommand{\col}[2]{{#1}_{\bullet,#2}}
\newcommand{\veps}{\varepsilon}

\newcommand{\gdw}{\Leftrightarrow}

\newcommand{\eins}{\mathds{1}}

\DeclareMathOperator{\rank}{rank}

\DeclareMathOperator{\Diag}{Diag}

\DeclareMathOperator{\supp}{supp}
\DeclareMathOperator{\conv}{conv}

\title[Tomographic Compressive Sensing]{
Average Case Recovery Analysis \\ of Tomographic Compressive Sensing}
\author[Petra, Schn\"{o}rr]{Stefania Petra, Christoph Schn\"{o}rr}
\address[Stefania Petra]{Image and Pattern Analysis Group, University of Heidelberg, Speyerer Str.~6, 69115 Heidelberg, Germany} 
\email{petra@math.uni-heidelberg.de}
\address[Christoph Schn\"{o}rr]{Image and Pattern Analysis Group, University of Heidelberg, Speyerer Str.~6, 69115 Heidelberg, Germany} 
\email{schnoerr@math.uni-heidelberg.de}
\date{} 

\thanks{Support by the German Research Foundation (DFG) is gratefully acknowledged, as part of the project ``3D-Tomographie mit wenigen Projektionen in der Experimentellen 3D-Str\"{o}mungsmessung'', grant SCHN457/11.}



\begin{document}

\sloppy

\begin{abstract}
The reconstruction of three-dimensional sparse volume functions from few tomographic projections constitutes a challenging problem in image reconstruction and turns out to be a particular instance
problem of compressive sensing. The tomographic measurement matrix encodes the incidence relation of the imaging process, and therefore is not subject to design up to small perturbations of non-zero entries. We present an average case analysis of the recovery properties and a corresponding tail bound to establish weak thresholds, in excellent agreement with numerical experiments. Our result improve the state-of-the-art of tomographic imaging in experimental fluid dynamics by a factor of three.
\end{abstract}

\maketitle

%
%\tableofcontents
%

%%%
\section{Introduction}

Research on compressive sensing \cite{CompressedSensing-06,Candes-CompressiveSampling-06} focuses on properties of underdetermined linear systems 
\begin{equation} \label{eq:Ax=b}
A x = b,\qquad A \in \R^{m \times n},\qquad
m \ll n,
\end{equation}
that ensure the accurate recovery of sparse solutions $x$ from observed measurements $b$. Strong assertions are based on random ensembles of measurement matrices $A$ and measure concentration in high dimensions that enable to prove good recovery properties with high probability \cite{L1LPSparseApproximate-06,ErrorCorrectingLP-05}.

A common obstacle in various application fields are the limited options for \emph{designing} a measurement matrix so as to exhibit desirable mathematical properties, are very limited. Accordingly, recent research has also been concerned with more restricted scenarios, spurred by their relevancy to applications (cf.~Section \ref{sec:related-work}).

Consequently, we consider a representative scenario, motivated by applications in experimental fluid dynamics (Fig.~\ref{fig:TomoPIV}). A suitable mathematical abstraction of this setup gives rise to a huge and severely underdetermined linear system \eqref{eq:Ax=b} that has additional properties: a \emph{very sparse} nonnegative measurement matrix $A$ with \emph{constant small support} of all column vectors, and a nonnegative sparse solution vector $x$:
\begin{equation} \label{eq:Axb-properties}
A \geq 0,\; x \geq 0,\qquad \supp(\col{A}{j}) = \ell \ll m,\qquad \forall j = 1,\dotsc,n.
\end{equation}
Our objective is the usual one: relating accurate recovery of $x$ from given measurements $b$ to the sparsity $k = \supp(x)$ of the solution $x$ and to the dimensions $m, n$ of the measurement matrix $A$. The sparsity parameter $k$ has an immediate physical interpretation (Fig.~\ref{fig:TomoPIV}). Engineers require high values of $k$, but are well aware that too high values lead to spurious solutions. The current practice is based on a rule of thumb leading to conservative low values of $k$. 

In this paper, we are concerned with working out a better compromise along with a mathematical underpinning. The techniques employed are general and only specific to the class of linear systems \eqref{eq:Ax=b}, \eqref{eq:Axb-properties}, rather than to a particular application domain.

\begin{figure}
\centerline{
\includegraphics[width=0.3\textwidth]{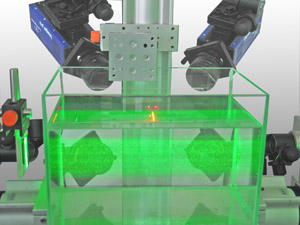} 
\hspace{0.025\textwidth}
\includegraphics[width=0.5\textwidth]{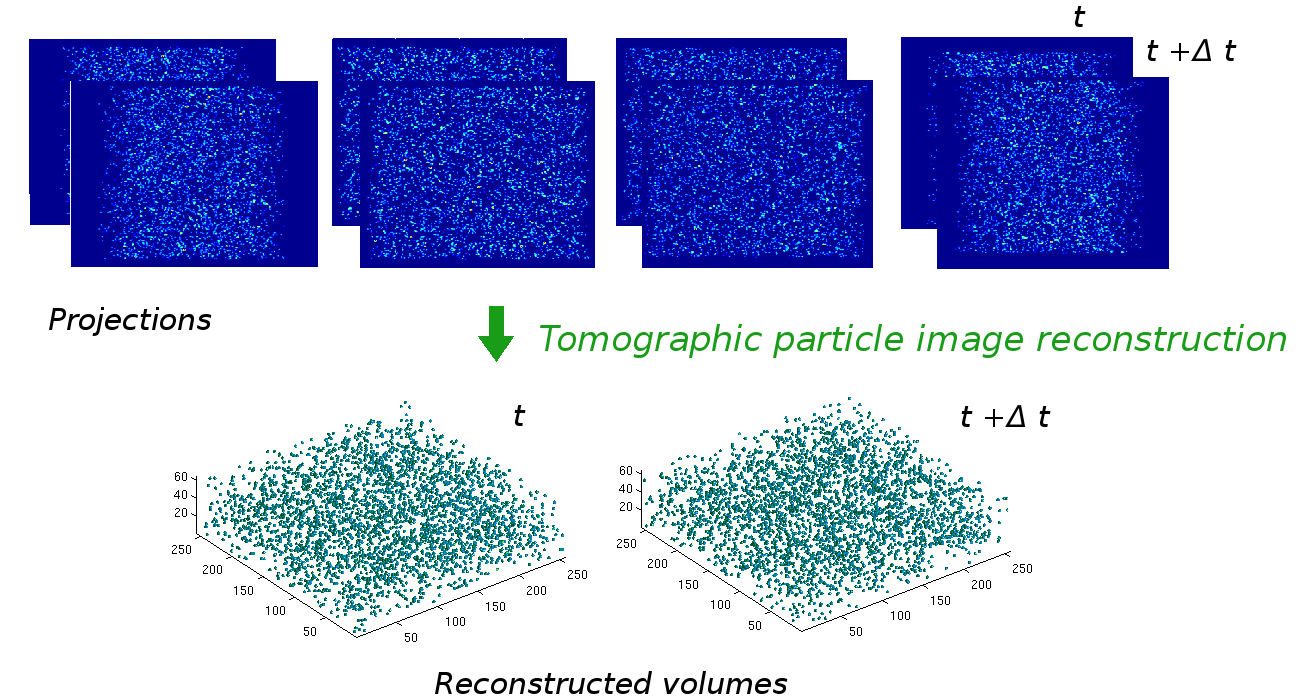}
}
\caption{
Compressive sensing in experimental fluid dynamics: A multi-camera setup gathers few projections from a sparse volume function. This scenario is described by a very large and highly underdetermined sparse linear system \eqref{eq:Ax=b} having the additional properties \eqref{eq:Axb-properties}. The sparsity parameter $k$ reflects the seeding density of a given fluid with particles. Less sparse scenarios increase the spatial resolution of subsequent studies of turbulent motions, but compromise accuracy of the reconstruction. Research is concerned with working out and mathematically substantiating the best compromise. 
}
\label{fig:TomoPIV}
\end{figure}

We regard the measurement matrix $A$ as \emph{given}. Concerning the design of $A$, we can only resort to small random perturbations of the non-zero entries of $A$, thus preserving the sparse structure that encodes the underlying incidence relation of the sensor. 
Additionally, we exploit the fact that solution vectors $x$ can be regarded as samples from a uniform distribution over $k$-sparse vectors, which represents with sufficient accuracy  the underlying physical situation.

Under these assumptions, we focus on an \emph{average case analysis} of conditions under which \emph{unique recovery} of $x$ can be expected with \emph{high probability}. A corresponding tail bound implies a weak threshold effect and criterion for adequately choosing the value of the sparsity parameter $k$. Our results are in excellent agreement with numerical experiments and improve the state-of-the-art by a factor of three.

\subsection*{Contribution and Organization} 
In Section \ref{sec:preliminaries}, we detail the mathematical abstraction of the imaging process and discuss directly related work. In Section \ref{sec:expanders}, we examine recent results of compressive sensing based on sparse expanders. This sets the stage for an average case analysis conducted in Section \ref{sec:weak-equivalence} and corresponding weak recovery properties, that are in sharp contrast to poor strong recovery properties presented in Section \ref{sec:strong}. We conclude with a discussion of quantitative results and their agreement with numerical experiments in Section \ref{sec:experiments}.

\subsection*{Notation}
$|X|$ denotes the cardinality of a finite set $X$ and $[n] = \{1,2,\dotsc,n\}$ for $n \in \N$. 
We will denote by
%$ \| \cdot \| $ and  $ \| \cdot \|_1 $ 
%the  Euclidean  $\ell_2$-norm and the $\ell_1$-norm in the $n$-dimensional real vector space $\R^n$.
$\|x\|_{0} = |\{i \colon x_{i} \neq 0 \}|$ %denotes the usual pseudo-norm 
and $\R_{k}^{n} = \{ x \in \R^{n} \colon \|x\|_{0} \leq k \}$ the set of $k$-sparse vectors. The corresponding sets of non-negative vectors are denoted by $\R_{+}^{n}$ and $\R_{k,+}^{n}$, respectively. The support
of a vector $x\in \R^{n}$, $\mrm{supp}(x) \subseteq [n]$,
is the set of indices of non-vanishing components of $x$. 
With $I^{+}(x) = \{i \colon x_{i} > 0\}$, $I^{0}(x) = \{i \colon x_{i} = 0\}$ and $I^{-}(x) = \{i \colon x_{i} < 0\}$, we have  $\mrm{supp}(x) = I^{+}(x) \cup I^{-}(x)$ and $\|x\|_{0} = |\mrm{supp}(x)|$.

For a finite set $S$, the set $\mc{N}(S)$ denotes the union of all
neighbors of elements of $S$ where the corresponding relation (graph)
will be clear from the context.

$\eins = (1,\dotsc,1)^{\T}$ denotes the one-vector of appropriate dimension.

$\col{A}{i}$ denotes the $i$-th column vector of a matrix $A$. For
given index sets $I, J$, matrix $A_{I J}$ denotes the submatrix of $A$ 
with rows and columns indexed by $I$ and $J$, respectively. $I^{c},
J^{c}$ denote the respective complement sets. Similarly, $b_{I}$
denotes a subvector of $b$.

$\EE[\cdot]$ denotes the expectation operation applied to a random variable and $\Pr(A)$ the probability to observe an event $A$.

\newpage

\section{Preliminaries}
\label{sec:preliminaries}
%%%
\subsection{Imaging Setup and Representation}
\label{sec:setup}

We refer to Figure \ref{fig_1} for an illustration of the mathematical abstraction of the scenario depicted by Figure \ref{fig:TomoPIV}.
In order to handle in parallel the 2D and 3D cases, we will use the variable
\begin{equation} \label{eq:def-D}
D \in \{2,3\}.
\end{equation}

We measure the \textbf{problem size} in terms of $d \in \N$ and consider $n:=d^D$ \textbf{cells} in a square ($D=2$) or cube ($D=3$)
and  $m:=Dd^{D-1}$ \textbf{rays}, compare Fig.~\ref{fig_1}, left and right. It will be useful to denote the set of cells by $C = [n]$ and the set of rays by $R = [m]$.
The incidence relation
between cells and rays is given by a 
$m\times n$ \textbf{measurement matrix} $A_d^D$ 
\begin{equation} \label{eq:def-AdD}
(A_d^D)_{ij}=
\begin{cases} 1,& 
\quad \text{if $j$-th ray intersects $i$-th cell},\\
 0, & \quad \text{otherwise}, \end{cases}
\end{equation}
for all $i\in[m]$, $j\in[n]$. Thus, cells and rays correspond to columns and rows of $A_d^D$.

The incidence relation encoded by $A_d^D$ gives rise to the equivalent representation in terms of a \textbf{bipartite graph} $G = (C,R;E)$ with left and right vertices $C$ and $R$, and edges $cr \in E$ iff $(A_{d}^{D})_{rc} = 1$. Figure \ref{fig_1} illustrates that $G$ has \textbf{constant left-degree} $\ell = D$. It will be convenient to use a separate symbol $\ell$. 

For a fixed vertex $i$, any adjacent vertex $j \sim i$ is called \textbf{neighbor} of $i$. For any non-negative measurement matrix $A$ and the corresponding graph, the set
\[
\calN(S) = \{i \in [m] \colon i \sim j,\, j \in S \} 
= \{i\in[m] \colon A_{ij}>0,\, j\in S\}
\]
contains all neighbors of $S$. The same notation applies to neighbors of subsets $S \subset [m]$ of right nodes.

With slight abuse, we call the matrix $A_{d}^{D}$ that encodes the adjacency $r \sim c$ of vertices $r \in R$ and $c \in C$ \textbf{adjacency matrix} of the induced bipartite graph $G$, deviating from the usual definition of the adjacency matrix of a graph that encodes the adjacency of \emph{all} nodes $v_{i} \sim v_{j},\, V = C \cup R$.
Moreover, in this sense, we will call any non-negative matrix adjacency matrix, based on its non-zero entries.

Let $A$ be the non-negative adjacency matrix of a bipartite graph with constant left degree $\ell$. The \textbf{perturbed matrix} $\tilde A$ is computed by uniformly perturbing the non-zero entries $A_{ij} > 0$ to obtain $\tilde A_{ij} \in [A_{ij}-\veps,A_{ij}+\veps]$, and by normalizing subsequently all column vectors of $\tilde A$. In practice, such perturbation can be implemented by discretizing the image by radial basis functions
and choose their locations on an irregular grid, see \cite{Petra2009}.

\vspace{0.25cm}
The following class of graphs plays a key role in the present context and in the field of compressed sensing in general.
\begin{definition}\label{def:Expander}
  A \textbf{$(\nu,\delta)$-unbalanced expander} is a bipartite simple graph $G
  = (L,R;E)$ with constant left-degree $\ell$ such that for any $X \subset L$ with
  $|X| \leq \nu$, the set of neighbors $\calN(X) \subset R$ of $X$ has at least size
  $|\calN(X)| \geq \delta \ell |X|$.
\end{definition}

\subsection{Deviation Bound}

We will apply the following inequalities for bounding the deviation of a random variable from its expected value based on martingales, that is on sequences of random variables $(X_{i})$ defined on a finite probability space $(\Omega, \mc{F}, \mu)$ satisfying
\begin{equation} \label{eq:condition-martingale}
 \EE[X_{i+1}|\mc{F}_{i}] = X_{i},\qquad
 \text{for all}\quad i \geq 1,
\end{equation}
where $\mc{F}_{i}$ denotes an increasing sequence of $\sigma$-fields in $\mc{F}$ with $X_{i}$ being $\mc{F}_{i}$-measurable.

This setting applies to random variables associated to measurements that are statistically \emph{dependent} due to the intersection of projection rays (cf.~Fig.~\ref{fig_1}).
\begin{theorem}[Azuma's Inequality \cite{Azuma1967,DasGupta2008}]
\label{thm:Azuma}
 Let $(X_{i})_{i=0,1,2,\dotsc}$ be a sequence of random variables such that for each $i$,
\begin{equation} \label{eq:def-ci}
 |X_{i}-X_{i-1}| \leq c_{i}.
\end{equation}
Then, for all $j \geq 0$ and any $\delta > 0$,
\begin{equation}
 \Pr\big(|X_{j}-X_{0}| \geq \delta\big) \leq 
 2 \exp\Big( -\frac{\delta^{2}}{2 \sum_{i=1}^{j} c_{i}^{2}}\Big).
\end{equation}
\end{theorem}

%The following theorem generalizes Theorem \ref{thm:Azuma} along the same line (cf.~\cite[Thm.~4.2]{Ledoux2001}).
%\begin{theorem} \label{thm:Ledoux-tailbound}
% Let $\mu = \mu_{1} \times \dotsb \times \mu_{s}$ be any product probability measure on the Cartesian product space $X = X_{1} \times \dotsb \times X_{s}$ of finite metric spaces $(X_{i},d_{i})$ equipped with the metric $d = \sum_{i=1}^{s} d_{i}$. Let $F \colon \mc{X} \to \R,\; F(X) = F(x^{1},\dotsc,x^{s})$, be a function satisfying 
%\begin{equation} \label{eq:Ledoux-condition}
% |F(x^{1},\dotsc,x^{i-1},x^{i},x^{i+1},\dotsc,x^{s}) - 
% F(x^{1},\dotsc,x^{i-1},y^{i},x^{i+1},\dotsc,x^{s})| < c_{i},
%\end{equation} 
%for $i=1,\dotsc,s$. Then
%\begin{equation} \label{eq:Ledoux-tailbound}
% \Pr\big(\{x \in X \colon |F - \int F d\mu| \leq \delta \}\big) 
% \leq 2 \exp\Big(-\frac{\delta^{2}}{2 \sum_{i=1}^{s} c_{i}^{2}}\Big), \qquad \forall \delta > 0.
%\end{equation}
%\end{theorem}

\begin{figure}
\centerline{
\includegraphics[width=0.35\textwidth]{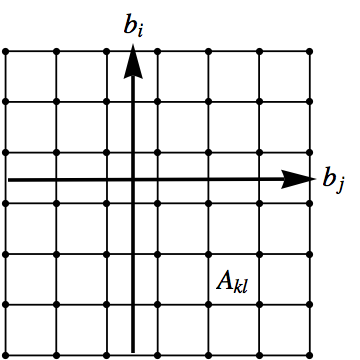}
\includegraphics[width=0.4\textwidth]{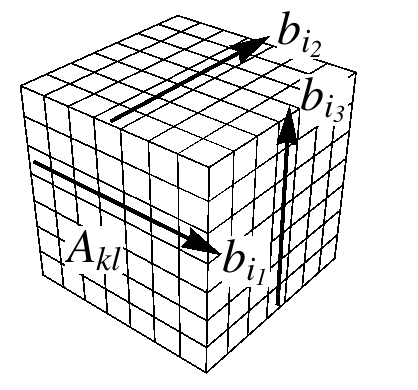}
}
\caption{
{\bf Left:} 
2D imaging geometry with $d^2$ cells and $2 d$ projection rays (here: $d=6$). The incidence relation is given by the measurement matrix $A = A_d^2$ (cf.~Eqn.~\eqref{eq:def-AdD}) which is the adjacency of a bipartite graph with constant left degree $\ell=2$.  
{\bf Right:} 
3D imaging geometry with $d^3$ cells and $3d^2$ rays (here: $d=7$). The incidence relation given by the measurement matrix 
$A = A_d^3$ is the adjacency of a bipartite graph with constant left degree $\ell=3$.
}
\label{fig_1}
\end{figure}

%%%
%
%%%

\subsection{Related Work}
\label{sec:related-work}
Although it was shown \cite{Candes-CompressiveSampling-06} that random measurement matrices are optimal for Compressive Sensing, in the sense that they require a minimal number of samples to recover efficiently a 
$k$-sparse vector, recent trends  \cite{RIP-P-SMM-08, XuHassibi_Expander} tend to replace
random dense matrices by adjacency matrices of ''high quality'' expander graphs.
Explicit constructions of such expanders exist, but are quite involved.
However, random $m\times n$ binary matrices with nonreplicative columns that have
$\lfloor \ell n\rfloor$ entries equal to $1$, perform numerically extremely well, even if $\ell$ is small, as shown in \cite{RIP-P-SMM-08}.
In \cite{HassibiIEEE} it is shown that perturbing the elements of
adjacency matrices of expander graphs with low expansion, can also improve performance.
This findings complement our prior work in \cite{Petra2009}, where
we observed that by slightly perturbing the entries of a tomographic projection
matrix its reconstruction performance can be improved significantly.

We wish to inspect the bounds on the required sparsity that
guarantee exact reconstruction of \emph{most} sparse signals, and 
corresponding critical parameter values similar to 
weak thresholds in \cite{DonTan05, DonohoT10}. The authors have computed
sharp reconstruction thresholds for Gaussian measurements, such that for given a signal 
length $n$ and numbers of measurements $m$, the maximal 
sparsity value $k$ which guarantees perfect reconstruction can be determined precisely.

For a matrix $A\in\R^{m\times n}$, Donoho and Tanner define the undersampling ratio
$\delta=\frac{m}{n}\in(0,1)$ and the sparsity as a fraction of $m$,
$k=\rho m$, for $\rho\in (0,1)$. The so called \emph{strong phase transition} 
$\rho_S(\delta)$ indicates the necessary undersampling ratio $\delta$ to recover  
\emph{all} $k$-sparse solutions, while the \emph{weak phase transition}
$\rho_W(\delta)$ indicates when $x^*$ with $\|x^*\|_0\le\rho_W(\delta) \cdot m$
can be recovered with overwhelming probability by linear programming. 

Relevant for TomoPIV is the setting as $\delta\to 0$ and
$n\to\infty$, that is severe undersampling, since the number of measurements is
of order $O(10^4)$ and discretization of the volume can be made accordingly fine.
For Gaussian ensembles a strong asymptotic threshold $\rho_S(\delta)\approx (2e
\log(1/\delta)^{-1}$ and weak asymptotic threshold
$\rho_W(\delta)\approx (2 \log(1/\delta)^{-1}$  holds, see e.g. \cite{DonTan05}.
In this highly undersampled regime, the asymptotic thresholds are the same
for nonnegative and unsigned signals.
Exact sparse recovery of nonnegative vectors has been also studied
in a series of recent papers \cite{HassibiIEEE, WangIEEE},
while \cite{Stojnic10a,Stojnic10b} additionally assumes that all nonzero elements 
are equal to each other.
As expected, additional information, improves the recoverable sparsity
thresholds. 

%??? deterministic construction of 0/1 matrix, 
%DeVore $m=\Omega(k^2)$, check again...

\subsubsection{Strong Recovery}

The maximal sparsity $k$ depending on $m$ and $n$, such that \emph{all} 
sparse signals are \emph{unique} and coincide with the
\emph{unique positive} solution of $Ax=b$, is investigated in
\cite{DonTan05, DonohoT10} from the perspective of
convex geometry by studying the face lattice of
the convex polytope  $\conv\{\col{A}{1},\dots,\col{A}{n},0\}$.
It is related to the nullspace property for nonnegative signals in
what follows.
\begin{theorem}[\cite{DonTan05, HassibiIEEE, WangIEEE, Petra2009}]
\label{thm:AllPOS} Let $A\in\Rmn$ be an arbitrary matrix.
Then the following statements are equivalent:
\begin{itemize}
\item[(a)] Every $k$-sparse nonnegative vector $x^*$ is the unique positive 
solution of $Ax=Ax^*$.
\item[(b)] The convex polytope defined as the convex hull of the columns
in $A$ and the zero vector, i.e. $\conv\{\col{A}{1},\dots,\col{A}{n},0\}$ is outwardly $k$-neighborly.
\item[(c)] Every nonzero null space vector has at least $k+1$
negative (and positive) entries.
\end{itemize}
\end{theorem}

\subsubsection{Weak Recovery}

Thm. 2 in \cite{DonTan05} shows the equivalence between
$(k,\epsilon)$-weakly (outwardly) neighborliness and
weak recovery, i.e. uniqueness of all except a fraction $\epsilon$
of $k$-sparse nonnegative vectors. Weak neighborliness is the same thing as saying that $A\Delta_0^{n-1}$ has at least $(1-\epsilon)$-times as many $(k-1)$-faces as the simplex $\Delta_0^{n-1}$. A different form of weak recovery is to determine the probability that
a random $k$-sparse positive vector by probabilistic nullspace analysis.
This concepts are related for an arbitrary sparse vector with exactly 
$k$ nonnegative entries in the next theorem.

\begin{theorem}\label{thm:individual_uniqueness_POS} 
Let $A\in\Rmn$ be an arbitrary matrix.
Then the following statements are equivalent:
\begin{itemize}
\item[(a)] The $k$-sparse nonnegative vector $x^*$ supported on
$S$, $|S|=k$, is the unique positive solution of $Ax=Ax^*$.
\item[(b)] Every nonzero null space vector cannot have all its
negative components in $S$.
\item[(c)] $A_S\R^k_+$ is a $k$-face of $A\R^n_+$, i.e. there exists a hyperplane
separating the cone generated by the linearly independent columns $\{\col{A}{j}\}_{j_\in S}$ from the cone generated by the columns of the off-support $\{\col{A}{j}\}_{j_\in S^c}$.
\end{itemize}
\end{theorem}
\begin{proof}
Statement (a) holds if and only if there is no $v\ne 0$ such that $Av=0$ and 
$v_{S^c}\ge 0$, compare for e.g. \cite[Thm. 1]{Man09ProbInteger}. Thus
(a) $\gdw$ (b). By \cite[Lem. 5.1]{DonohoT10}, (a) $\gdw (c)$ holds as well.
\end{proof}
%%%
If, in addition, all $k$ nonzero entries are equal to each other, 
then a stronger characterization holds.
%%%
\begin{theorem}[{\cite[Prop.~2]{Man09ProbInteger}}]
%\label{prop:LP-uniqueness}
\label{thm:individual_uniqueness_BIN} Let $A\in\Rmn$ be an arbitrary matrix.
Then the following statements are equivalent:
\begin{itemize}
\item[(a)] The $k$-sparse binary vector $x^*\in\{0,1\}^n$ supported on
$S$, $|S|=k$, is the unique
solution of $Ax=Ax^*$ with $x\in[0,1]^n$.
\item[(b)] Every nonzero null space vector cannot have all its
negative components in $S$ and the positive ones in $S^c$. 
\item[(c)] There exists a vector $r$ such that $\Diag(z^*)A^\top r > 0$, with
$z^*:=e-2 x^*$.
\item[(d)] $0 \in \R^{m}$ is not contained in the convex 
hull of the columns of $A\Diag(z^*)$, i.e. $0\notin\conv\{z^*_1\col{A}{1},\dots,
z^*_n\col{A}{n},0\}$, with $z^*:=e-2 x^*$.
\end{itemize}
\end{theorem}
\begin{proof}
If $x^*$ is unique in $\{0,1\}^n$, it is unique in $[0,1]^n$ as well.
Uniqueness in $[0,1]^n$ holds, for e.g. by \cite[Thm. 1]{Man09ProbInteger},
if there is no $v\ne 0$ such that $Av=0$, $v_{S^c}\ge 0$ and $v_{S}\le 0$, which shows
equivalence to (b).
With $D:=\Diag(e-2 x^*)$ and $DD=I$, (b) can be rewritten as follows:
there is no $v\ne 0$ such that $ADDv=0, Dv\ge 0$, $Dv\ne 0$.
With $u:=Dv$, the above condition becomes:
$$
ADu=0, u\ge 0, u\ne 0\ , \rm{has\ no\ solution}\ ,
$$
which by  Gordon's theorem of alternative  gives the equivalent certificate (c):
\begin{equation}\label{eq:Mangasarian-primal}
\exists r \quad  {\rm such\ that\ } D A^\top r > 0 \ .
\end{equation}
In other words, a small $k$-subset of the columns of $A$, are ''flipped'' by multiplication with $-1$, and these modified columns together with all remaining ones can be separated from the origin, which shows equivalence to (d), i.e. $0$ is not contained in the convex hull of these points.
\end{proof}

Note that statement (d) is related to the necessary condition for uniqueness
in \cite[Thm. 1]{WangIEEE}. We further comment on Thm. \ref{thm:individual_uniqueness_BIN} (c) from a probabilistic viewpoint. Condition (c) says that all points defined by the columns of $A\Diag(e-2 x^*)$ are located in a single half space defined by a hyperplane through the origin with normal $r$. Conditions under which this is likely to hold were studied by Wendel \cite{Wendel-62}. This problem is also directly related to the basic pattern recognition problem concerning the linear classification\footnote{In this context, ``linear'' means affine decision functions.} of any dichotomy of a finite point set \cite{CoverSeparability-65}.

Assuming $n$ points in $\R^{m}$ to be in general position, that is any subset of $m$ vectors is linearly independent, and that the distribution from which the given point set is regarded as an i.i.d.~sample set is symmetric with respect to the origin, then condition \eqref{eq:Mangasarian-primal} holds with probability
\begin{equation} \label{eq:WendelPR}
\Pr(n,m) = \frac{1}{2^{n-1}} \sum_{i=0}^{m-1} 
\binom{n-1}{i}.
\end{equation}
As Figure \ref{fig:WendelPR} illustrates, $\Pr(n,m) = 1$ if $n/m \leq 1$, due to the well known fact that any dichotomy of $m+1$ points in $\R^{m}$ can be separated by a hyper-plane \cite{Vapnik1971,Devroye1996}. For increasing dimension $m \to \infty$, this also holds almost surely if $n/m < 2$, which can be easily deduced by applying a binomial tail bound.
Accordingly, assuming that the measurement matrix $A$ conforms to the assumptions, the authors of 
\cite{Man09ProbInteger} conclude that an existing binary solution to \eqref{eq:Ax=b}  is unique with probability \eqref{eq:WendelPR} for \emph{underdetermined} systems with ratio $m/n > 1/2$.

We adopt this viewpoint in Section \ref{sec:recovery-perturbed} and develop a criterion for unique recovery with high probability using the \emph{given} measurement matrix \eqref{eq:def-AdD}, based on a probabilistic average case analysis of condition \eqref{eq:Hassibi-condition} (Section \ref{sec:reduced-system}). This criterion currently characterizes best the design of tomographic scenarios (Fig.~\ref{fig_1}), with recovery performance guaranteed with high probability. We conclude this section by mentioning that \emph{exact
nonasymptotic} recovery results for a $k$-sparse nonnegative vector are obtained in \cite[Thm. 1.10]{DonohoT10} by exploiting Wendel's theorem. Donoho and Tanner show that the probability of uniqueness of a $k$-sparse nonnegative vector equals $\Pr(n-m,n-k)$, provided $A$ satisfies certain conditions which do not hold in
our considered application.   

%%%
\begin{figure}
\centerline{
\includegraphics[width=0.5\textwidth]{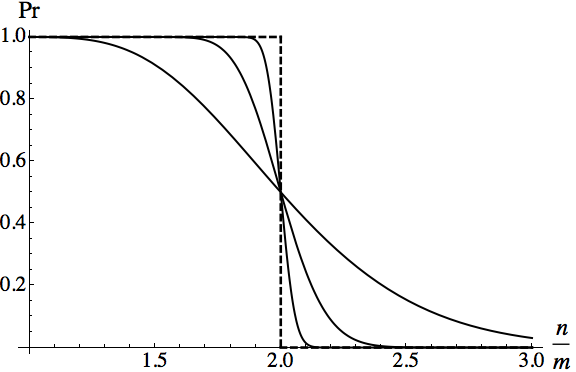}
}
\caption{
The probability $\Pr(n,m)$ given by \eqref{eq:WendelPR} that $n$ points in general position in $\R^{m}$ can be linearly separated \cite{Wendel-62}. This holds with probability $\Pr(n,m)=1$ for $n/m \leq 1$, and with $\Pr(n,m) \to 1$ if $m \to \infty$ and $1 \leq n/m < 2$.
}
\label{fig:WendelPR}
\end{figure}
%%%

%%%
%
%%%
\section{Expanders, Perturbation, and Weak Recovery}
\label{sec:expanders}

This section collects recent results of recovery properties based on expanders associated with sparse measurement matrices, possibly after a random perturbation of the non-zero matrix entries. Section \ref{sec:wang-hassibi-application} applies these results to our specific setting in a form suitable for a probabilistic analysis of recovery performance presented in Section \ref{sec:weak-equivalence}.

\subsection{Expanders and Recovery}
The following theorem is a slight variation of Theorem 4 in \cite{WangIEEE} tailored to our specific setting.
\begin{theorem}\label{thm:wang} Let $A$ be the adjacency
matrix of a $(\nu,\delta)$-unbalanced expander 
and $1 \geq \delta>\frac{\sqrt{5}-1}{2}$.
Then for any $k$-sparse vector $x^*$ with $k\le \frac{\nu}{(1+\delta)}$, the
solution set $\{x \colon Ax=Ax^*,x \ge 0\}$ is a singleton.
\end{theorem}
%%%%%%
\begin{proof} We will show that every nonzero null space vector has 
\emph{at least} $\frac{\nu}{(1+\delta)}+1$ negative and positive entries. Then Theorem \ref{thm:AllPOS} will provide the desired assertion.

Suppose without loss of generality that there is a vector $v\in\ker({A})\setminus \{0\}$
with 
\begin{equation}\label{eq:ker1}
s:=|I^{-}(v)|\le \frac{\nu}{(1+\delta)} \ .
\end{equation}
Then
\begin{equation}\label{eq:ker2}
\ell |I^{-}(v)| \ge |\calN(I^{-}(v))|\ge \delta \ell s, 
\end{equation}
where the second inequality follows by assumption due to the expansion property. 

Denoting by $S$ the support of $v$,  $S=I^{-}(v)\cup I^{+}(v)$, 
we have
\begin{equation}\label{eq:equalNeigh}
\calN(I^{-}(v))=\calN(I^{+}(v))=\calN(S) \ ,
\end{equation}
since otherwise $ Av\ne 0$ because $A$ is non-negative. 

From $\ell |I^{+}(v)|\ge |\calN(I^{+}(v))|$, \eqref{eq:equalNeigh}
and \eqref{eq:ker2},
we obtain
\begin{equation}\label{eq:ker3}
|I^{+}(v)| \ge \delta s \ . 
\end{equation}

Thus,
\begin{equation}\label{eq:ker33}
|S|=|I^{-}(v)| +|I^{+}(v)| \ge 2 \delta s \geq (1+\delta) s. 
\end{equation}

Let $\tilde S \subseteq S$ such that $|\tilde S|=\lfloor (\delta +1) s \rfloor$.
Thus $|\tilde S|\le \nu$ and
\begin{equation}\label{eq:ker4}
|\calN(\tilde S)|\ge \delta \ell |\tilde S|\ge \delta \ell (\delta +1) s
> s\ell\ 
\end{equation}
provided $\delta (1+\delta) > 1 \;\gdw\; \delta > (\sqrt{5}-1)/2$.
Summarizing, we get
$s\ell<|\calN(\tilde S)|\leq|\calN(S)|=|\calN(I^{-}(v))|\le s\ell$,
hence a contradiction.
\end{proof}
The assertion of Theorem \ref{thm:wang} solely relies on the expansion property of the measurement matrix $A$. Theorem \ref{thm:SP} below will be based on it and in turn the results of Section \ref{sec:recovery-unperturbed}.

%%%
\subsection{Perturbed Expanders and Recovery}
\label{sec:Hassibi}
We describe next an alternative route based on the \textbf{complete (Kruskal) rank} $r_{0} = r_{0}(A)$ of a measurement matrix $A$. This is the maximal integer $r_{0}$ such that every subset of $r_{0}$ columns of $A$ is linearly independent.

While this number is combinatorially difficult to compute in practice, both the number and the corresponding recovery performance can be enhanced by relating it to a particular expansion property of the bipartite graph associated to a \emph{perturbed} measurement matrix $\tilde A$. The latter can be easily computed in practice while preserving its sparsity, i.e.~the constant left-degree $\ell$.
%%%
\begin{theorem}[{\cite[Thm.~6.2]{Petra2009}, \cite[Thm.~4.1]{HassibiIEEE}}]\label{thm:Hassibi-4-1}
 Let $A$ be a non-negative matrix with $\ell$ non-zero entries in each column and complete rank $r_{0}=r_{0}(A)$. Then $|I^{-}(v)| \geq r_{0}/\ell$ for all nullspace vectors $v \in \ker(A)$.
\end{theorem}
\begin{remark} \label{rem:Hassibi-recovery}
In view of Theorem \ref{thm:AllPOS}, (c), Theorem \ref{thm:Hassibi-4-1} says that all $k$-sparse non-negative vectors $x$ can be uniquely recovered if $k \leq \lceil r_{0}/\ell - 1 \rceil$.
\end{remark}
The following Lemma asserts that by a perturbation of the measurement matrix the complete rank, and hence the recovery property, may be enhanced provided all subsets of columns, up to a related cardinality, entail an expansion that is \emph{less} however than the one required by Theorem \ref{thm:wang}.
\begin{lemma}[{\cite[Lemma 4.2]{HassibiIEEE}}]\label{lem:Hassibi-4-2}
Let $A$ be a non-negative matrix with $\ell$ non-zero entries in each column. Suppose that for a submatrix formed by $\tilde r_{0}$ columns of $A$ it holds that $|\mc{N}(X)| \geq |X|$, for each subset $X \subset C$ of columns of cardinality $|C| \leq \tilde r_{0}$, and with respect to the bipartite graph induced by $A$. Then there exists a perturbed matrix $\tilde A$ that has the same structure as $A$ such that its complete rank satisfies $r_{0}(\tilde A) \geq \tilde r_{0}$.
\end{lemma}
Theorem \ref{thm:SPCS1} below and Section \ref{sec:recovery-perturbed} will be based on Theorem \ref{thm:Hassibi-4-1} and Lemma \ref{lem:Hassibi-4-2}.

%%%
\subsection{Weak Reconstruction Guarantees}
\label{sec:wang-hassibi-application}

We introduce some further notions used subsequently to state our results. Let $A$ denote the matrix $A_{d}^{D}$ defined by \eqref{eq:def-AdD}, and consider a subset $X \subset C$ of $|X|=k$ columns and a corresponding $k$-sparse vector $x$. Then $b = A x$ has support $\mc{N}(x)$, and we may remove the subset of $\mc{N}(X)^{c} = (\mc{N}(X))^{c}$ rows from the linear system $A x = b$ corresponding to $b_{r}=0,\, \forall r \in R$. Moreover, based on the observation $\mc{N}(X)$, we know that
\begin{equation}\label{eq:reduced-dimensions}
 X \subseteq \mc{N}(\mc{N}(X))
 \qquad\text{and}\qquad
 \mc{N}(\mc{N}(X)^{c}) \cap X = \emptyset \ .
\end{equation}
Consequently, we can restrict the linear system $A x = b$ to the subset of columns
$\mc{N}(\mc{N}(X)) \setminus \mc{N}(\mc{N}(X)^{c}) \subset C$. This will be detailed below by Proposition \ref{prop:redfeasSet}.

\vspace{0.5cm}
In practical applications, the reconstruction of a random $k$-sparse vector $x$ will be based on a reduced linear system with the above dimensions. These dimensions will be the same for \emph{all} random sets $X = \supp(x)$ contained in $\mc{N}(\mc{N}(X))$. Consequently, in view of a probabilistic average case analysis conducted in Section \ref{sec:weak-equivalence}, it suffices to measure the expansion with respect to these sets.

Taking this into account, the following theorem tailors Theorem \ref{thm:wang} to our specific setting.
\begin{theorem}\label{thm:SP}  Let $A$ be the adjacency
matrix of a bipartite graph such that for all
random subsets $X \subset C$ of $|X| \leq k$ left nodes, the set of neighbors $\calN(X)$ of $X$ satisfies
\begin{equation} \label{eq:condition-Wang}
|\calN(X)| \geq \delta 
\ell |\calN(\calN(X))\setminus \calN(\calN(X)^c)| 
\qquad\text{with}\qquad 
\delta>\frac{\sqrt{5}-1}{2}.
\end{equation}
Then, for any $k$-sparse vector $x^*$, the
solution set $\{x \colon  Ax=Ax^*,x \ge 0\}$ is a singleton.
\end{theorem}

Likewise, the following theorem applies the statements of Section \ref{sec:Hassibi} to our specific setting.
\begin{theorem}\label{thm:SPCS1}  Let $A$ be the adjacency
matrix of a bipartite graph such that for all
subsets $X \subset C$ of $|X| \leq k$ left nodes, the set of neighbors $\calN(X)$ of $X$ satisfies
\begin{equation} \label{eq:Hassibi-condition}
|\calN(X)| \geq \delta 
\ell |\calN(\calN(X))\setminus \calN(\calN(X)^c)| 
\qquad\text{with}\qquad 
\delta > \frac{1}{\ell}.
\end{equation}
Then, for any $k$-sparse vector $x^*$, there exists a perturbation $\tilde A$ of $A$ such that the
solution set $\{x \colon \tilde Ax=\tilde Ax^*,x \ge 0\}$ is a singleton.
\end{theorem}
The consequences of Theorems \ref{thm:SP} and \ref{thm:SPCS1} are investigated in Section \ref{sec:weak-equivalence} by working out critical values of the sparsity parameter $k$ for which the respective conditions are satisfied with high probability.

\section{Strong Equivalence}
\label{sec:strong}

In \cite{Petra2009} we tested the properties of the discrete tomography matrix
in focus against various conditions, like the null space property, the restricted isometry property, etc., and predicted an extremely poor worst case performance of such a measurement
system. In the 3D case we showed that the strong threshold on sparsity, that is the maximal
sparsity level $k_0$ for which recovery of \emph{all} $k$-sparse (positive) vectors,
$k\le k_0$, is guaranteed, is a constant, not depending on the undersampling ratio $d$.

\subsection{Unperturbed Systems}

Given an indexing of cells and rays, we can rewrite the 
projection matrix $A^D_d\in\R^{Dd^{D-1} \times d^D}$ from \eqref{eq:def-AdD}
in closed form as
\begin{equation}\label{eq:A_dD}
 A^D_d:=\begin{cases}\left( \begin{array}{c}
   I_d\otimes \eins_d^T\\
   \eins_d^T\otimes I_d \\
\end{array} \right) ,  & \text{if }D=2\ ,\\[5mm]
\left( \begin{array}{c}
   \eins_d^\T \otimes I_d \otimes I_d \\
    I_d\otimes \eins_d^\T \otimes I_d\\
    I_d\otimes I_d\otimes \eins_d^\T 
   \end{array} \right)  , & \text{if } D=3 \ .
\end{cases}
\end{equation}
Since for this matrices a sparse nullspace basis can be computed, we
can derive the maximal sparsity via the nullspace property, as shown
next.
\begin{proposition}\cite[Prop. 2.2, Prop. 3.2]{Petra2009}\label{prop:RankKernelA}  Let 
$D\in\{2,3\}$, $d\in \N$, $d\ge 3$ and $A^D_d$ from \eqref{eq:A_dD}. Define $B^D_d\in\R^{d^D\times (d-1)^D}$ as
\begin{equation}\label{eq:B_D}
    B^D_d:=\begin{cases}
      \left(\begin{array}{c} -\eins_{d-1}^T\\ I_{d-1}\end{array} \right)
       \otimes
       \left(\begin{array}{c} -\eins_{d-1}^T\\ I_{d-1}\end{array} \right)
  ,  & \text{if }D=2,\\[5mm]
   \left(\begin{array}{c} -\eins_{d-1}^\T \\ I_{d-1}\end{array} \right)
       \otimes
       \left(\begin{array}{c} -\eins_{d-1}^\T \\ I_{d-1}\end{array} \right)
       \otimes 
       \left(\begin{array}{c} -\eins_{d-1}^\T \\ I_{d-1}\end{array} \right), & \text{if } D=3
\ .
\end{cases}
      \end{equation}
Then the following statements hold
\begin{itemize}
\item[(a)] $A^D_d B^D_d=0$.
\item[(b)] Every column in $B^D_d$ has exactly $2^D$ nonzero ($2^{D-1}$ positive, 
$2^{D-1}$ negative) elements.
\item[(c)] $B^D_d$ is a full rank matrix and $\rank(B^D_d)=(d-1)^D$.
\item[(d)] $\ker(A^D_d)=span\{B^D_d\}$, i.e. the columns of $B^D_d$ provide a basis 
           for the null space of $A^D_d$.
\item[(e)] $\rank(A^D_d)=d^D-(d-1)^D$.
\item[(f)] $\sum_{i=1}^nv_i=0$ holds for all $v\in\ker(A^D_d)$.
\item[(g)] The Kruskal rang of $A^D_d$ is $2^D-1$, i.e.
$$\min_{\substack{v\in\ker(A^D_d)\\ v\ne 0}} \|v\|_0=2^D\ .$$
\item[(h)] Every nonzero nullspace vector has at least $2^{D-1}$ negative
entries. i.e.
$$\min_{\substack{v\in\ker(A^D_d)\\ v\ne 0}} |I^{-}(v)|=2^{D-1}\ .$$
\end{itemize}
\end{proposition}
Thus, (g) and (h) imply
\begin{corollary}\label{cor:strong} For all $d\in \N$, $d\ge 3$, every 
$\left(2^{D-1}-1\right)$-sparse vector $x^{\ast}$ is the unique sparsest solution of $A^D_d x =
A^D_d x^{\ast}$. Moreover, for every $\left(2^{D-1}-1\right)$-sparse positive vector $x^*$ 
$\{x \colon A^D_d x = A^D_d x^*\}$ is a singleton.
\end{corollary}
This bound is tight, since we can construct two $2^{D-1}$-sparse solutions $x^1$
and $x^2$ such that $A^D_d x^1=A^D_d x^2$, compare Fig. \ref{fig:NonUnique} for the
3D case.
However, when D=3, not every 8-column combination, or more, in $A^3_d$ is linearly dependent.
In fact, only a limited number of $k$-column combinations
can be dependent without violating $\rank(A^3_d)=3d^2-3d+1$.
It turns out that this number is tiny for smaller $k$ when compared to 
$\binom{n}{k}$. As $k$ increases this number also grows
and equals 1 only when $k>\rank(A^3_d)$.
Likewise, not every 4-sparse binary vector is nonunique. Due to the simple geometry 
of the problem it is not difficult to count the ''bad'' 4-sparse configurations in 3D.
Since they are always located in 4 out of 8 corners of a cuboid in the $d^3$ cube, compare Fig. \ref{fig:NonUnique} left, and there are only two possibilities to choose them, the probability that a 4-sparse binary vector is unique, equals
\begin{equation*}
1-\frac{2 \binom{d}{2}^3}{\binom{d^3}{4}}=
1-\frac{ 6 (d-1)^2 }{(d^2+ d+1)(d^3-2)(d^3-3)}=
1-\mathcal{O}(d^{-6})\xrightarrow{d \to \infty} 1\ .
\end{equation*}

%%%%%%%%%%%%%
\begin{figure}
\begin{tabular}{c c}
\includegraphics[clip,width=0.20\textwidth]{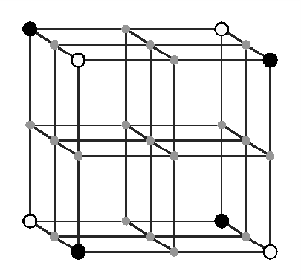}&
\includegraphics[clip,width=0.30\textwidth]{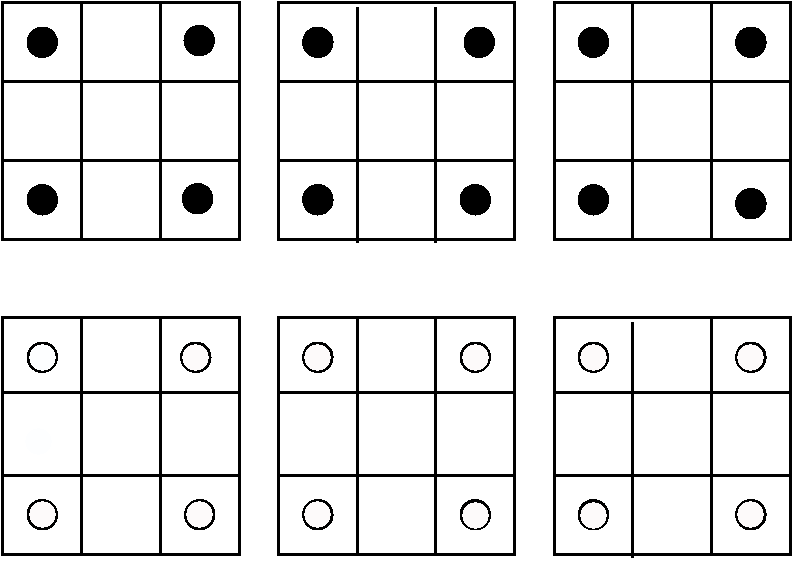}
\end{tabular}
\caption{ Two different {\em non unique} $4$-sparse ''particle'' distributions
in a $3\times 3\times 3$ volume. 
Both configurations (represented by black and white dots) yield 
identical projections in all three directions.  Such nonunique configurations correspond to
positive or negative entries in an $8$-sparse nullspace vector of $A^3_d$,
compare to Prop.  \ref{prop:RankKernelA}.}
\label{fig:NonUnique}
\end{figure}
%%%%%%%%%%%%%
\begin{figure}
\begin{tabular}{ccc}
\includegraphics[height=0.23\textheight]{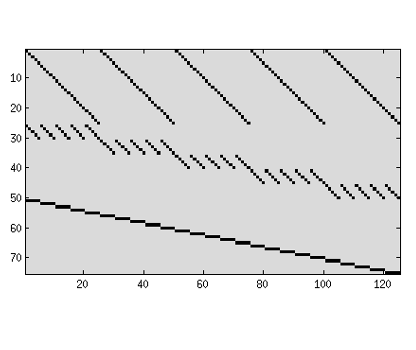} &
\includegraphics[height=0.23\textheight]{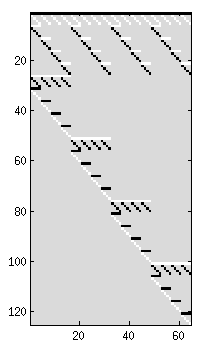}&
\includegraphics[height=0.23\textheight]{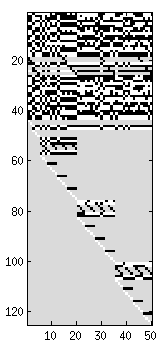}
\end{tabular}
\caption{{\bf Left:} The 3D projection matrix $A^3_5$ and,  {\bf middle,} a sparse basis which spans its nullspace. {\bf Right:} If we allow a small perturbation of the nonzero entries of $A^D_d$, all corresponding nullspace vectors of the perturbed 
matrix will be less sparse and lie in a $d^{D-1}(d-D)$-dimensional subspace, compared to
$(d-1)^D$ in the unperturbed case.}
\label{fig_nullspace}
\end{figure}

\subsection{Perturbed Systems}
%%%
The weak performance of $A^D_d$ rests upon its small Kruskal rank.
In order to increase the maximal number $k$ of columns such
that all $k$ (or less) column combinations are linearly independent we 
perturb the nonzero entries of the original matrix $A^D_d$.
Figure \ref{fig_nullspace}, right, indicates that perturbation
leads to less sparse nullspace vectors. 
If we could estimate the Kruskal rank $\tilde r_0$ of the perturbed system 
we could apply Thm. \ref{thm:Hassibi-4-1} and obtain a lower bound on the sparsity
yielding strong recovery \emph{for all} $\lceil \tilde r_{0}/\ell - 1 \rceil$-sparse vectors.
However, determining $\tilde r_0$ for the perturbed matrix seems impossible.
We believe however that it increases with $d$, in contrast to the constant $2^D-1$ in case of
unperturbed systems. 
Luckily, it will turn out in Section \ref{sec:recovery-unperturbed} that the weak recovery threshold for unperturbed systems will give a \emph{lower bound} on the strong recovery threshold for
perturbed matrices, since reduced systems will be strictly overdetermined and guaranteed to have full rank.

%%%%%%%%%%%
\section{Weak Recovery}\label{sec:weak-equivalence}
%%%
In this section, we consider the recovery properties of the 3D setup depicted in Fig.~\ref{fig_1} and establish conditions for weak recovery, that is conditions for unique recovery that holds \emph{on average with high probability}. We clearly point out that our conditions do \emph{not} guarantee unique recovery in \emph{each} concrete problem instance. 
\begin{remark} \label{rem:probability}
In what follows, the phrase \textbf{with high probability} refers to values of the sparsity parameter $k$ for which random supports $|\supp(b)|$ concentrate around the crucial expected value $N_{R}$ according to Prop.~\ref{prop:NR0-deviation}, thus yielding a desired threshold effect.
\end{remark}

We first inspect in Section \ref{sec:reduced-system} the effect of sparsity on the expected dimensions of a reduced system of linear equations, along with its equivalence to the original system. Subsequently, we establish the aforementioned conditions 
based on Theorems \ref{thm:SP} and \ref{thm:SPCS1}, and on the \emph{expected} quantities involved in the corresponding conditions.

In particular, we establish such uniqueness conditions for reduced underdetermined systems of dimension $m/n > (\sqrt{5}-1)/2 \approx 0.618$. Our results are in excellent agreement with numerical experiments discussed in Section \ref{sec:experiments}.

\subsection{Reduced System}\label{sec:reduced-system}
%%%
We formalize the system reduction described in %connection with 
Eqn.~\eqref{eq:reduced-dimensions}. Besides checking its equivalence to the unreduced system, we compute the expected reduced dimensions together with a deviation bound. Additionally, we determine critical values of the sparsity parameter $k$ that lead to overdetermined reduced systems.

Recall from Section \ref{sec:setup} that we regard a given measurement matrix $A$ also as adjacency matrix of a bipartite graph $G = (C,R;E)$.
%The notation introduced will also economize the presentation in subsequent sections.

\subsubsection{Definition and Equivalence}
\begin{definition}
The \textbf{reduced system} corresponding to a given non-negative vector $b$,
\begin{equation} \label{eq:red-system}
 A_{red} x = b_{red},\qquad A_{red} \in 
 \R_{+}^{m_{red} \times n_{red}},
\end{equation}
results from $A, b$ by choosing the subsets of rows and columns
\begin{equation} \label{eq:def-RbCb}
 R_{b} := \supp(b),\qquad
 C_{b} := \mc{N}(R_{b}) \setminus \mc{N}(R_{b}^{c})
\end{equation}
with 
\begin{equation} \label{def:mn-red}
 m_{red} := |R_{b}|,\qquad
 n_{red} := |C_{b}|.
\end{equation}
\end{definition}
Note that for a vector $x$ and the bipartite graph induced by the measurement matrix $A$, we have the correspondence (cf.~\eqref{eq:reduced-dimensions})
\[
X = \supp(x),\qquad
R_{b} = \mc{N}(X),\qquad
C_{b} = \mc{N}(\mc{N}(X)) \setminus \mc{N}(\mc{N}(X)^{c}).
\]
We further define
\begin{equation}\label{def:feasSet}
\calS^+:=\{x \colon Ax=b, x\ge 0\}
\end{equation}
and
\begin{equation}\label{def:redfeasSet}
\calS_{red}^+:=\{x \colon A_{R_b C_b}x=b_{R_b}, x\ge 0\}\ .
\end{equation}
The following proposition asserts that solving the reduced system 
\eqref{eq:red-system} will always recover the support of the solution to the original system $A x  = b$.
\begin{proposition}\label{prop:redfeasSet}
Let $A\in \R^{m\times n}$ and $b\in\R^m$ have nonnegative entries only, and let
$\calS^+$ and $\calS_{red}^+$ be defined by \eqref{def:feasSet} and
\eqref{def:redfeasSet}, respectively. Then
\begin{equation}\label{eq:feasSet}
  \calS^+=\{x\in \R^n \colon x_{(C_b)^c}=0\;\text{ and }\; x_{C_b}\in
  \calS_{red}^+\}.
\end{equation}
\end{proposition}
\begin{proof}
Let $S:=\{x\in \R^n \colon x_{(C_b)^c}=0\text{ and } x_{C_b}\in \calS^+_{red}\}$. We first show $S\subseteq \calS^+$. Let $x\in S$. From this
$x\ge 0$ follows directly. We thus just have to show $\sum_{j=1}^n
a_{ij}x_j=b_i, \forall i\in [n]$. Indeed, for
\begin{equation*}
i\in R_b :\quad
\sum_{j=1}^n a_{ij}x_j=
\sum_{j\in C_b}\underbrace{a_{ij}x_j}_{=b_i}
+\sum_{j\in (C_b)^c}a_{ij}\underbrace{x_j}_{=0}=b_i\ ,
\end{equation*}
whereas for
\begin{equation*}
i\in (R_b)^c :\quad
\sum_{j=1}^n a_{ij}x_j=
\sum_{j\in C_b}\underbrace{a_{ij}}_{=0}x_j
+\sum_{j\in (C_b)^c}\underbrace{a_{ij}}_{>a_{ij}}
\underbrace{x_j}_{=0}=0=b_i\ .
\end{equation*}
Now let $x\in\calS^+$ and consider any $i\in (R_b)^c$. Then
\begin{equation}\label{eq:x_(C_b)^c==0}
0=b_i=\sum_{j=1}^n a_{ij}x_j=\sum_{j\in C_b}\underbrace{a_{ij}}_{=0}x_j +\sum_{j\in
(C_b)^c}\underbrace{a_{ij}}_{>a_{ij}}x_j
\end{equation}
holds. Since $x\ge 0$, we obtain from \eqref{eq:x_(C_b)^c==0} that
$x_j=0,\forall j\in (C_b)^c$. To show that $A_{R_b C_b}x_{C_b}=b_{R_b}$, consider
\begin{equation*}
i\in R_b:\quad \sum_{j\in C_b}a_{ij}x_j =\sum_{j\in C_b}a_{ij}x_j+ \sum_{j\in
(C_b)^c}a_{ij}\underbrace{x_j}_{=0}=\sum_{j=1}^n a_{ij}x_j=b_i\ .
\end{equation*}
Hence, $x_{(C_b)^c}=0$ and $x_{C_b }\in\calS^+_{red}$. Thus $x\in S$.
\end{proof}

In the following two sections, we compute the expected values of the reduced system dimension \eqref{def:mn-red}.

%%%%%%%%%%%%
\subsubsection{Expected Number of Non-Zero Measurements}

We consider the uniform random assignment of $k$ particles to the $n = |C|$ cells $c \in C$. A single cell may be occupied by more than a single particle. This corresponds to the physical situation that real particles are very small relative to the discretization depicted by Figure \ref{fig_1}. The imaging optics enlarges the appearance of particles, and the action of physical projection rays is adequately represented by linear superposition.

This scenario gives rise to a random vector $x \in \R_{k,+}^{n}$ with support $|\supp(x)| \leq k$. It generates a vector 
\begin{equation}
b = A_{d}^{D} x \in \R_{+}^{m}
\end{equation}
of measurements. We are interested in the expected size of the support of $b$, 
\begin{equation} \label{eq:def-NR}
N_{R} := \EE[|\supp(b)|],\qquad
N_{R}^{0} := m - N_{R},
\end{equation}
that equals the number of projection rays $r \in R$ with non-vanishing measurements $b_{r} \neq 0$. We denote the event $b_{r} = 0$ by the binary random variable\footnote{We economize notation here by re-using the symbol $X$, a random indicator vector indexed by rays (right nodes) $r \in R$. Due to the context, there should be no danger of confusion with $X = \supp(x)$ denoting random subsets of left nodes used in other sections.} $X_{r}=1$, i.e.~$X_{r}=0$ corresponds to the event $b_{r} > 0$ that at least a single particle meets ray $r$.

The probability that a single $c$ is met by ray $r$ is 
\begin{equation} \label{eq:def-qd}
q_{d} := \frac{d}{|C|} = \frac{d}{n} = \frac{1}{d^{D-1}}.
\end{equation}
For $k$ particles, the probability that $0 \leq i \leq k$ particles meet projection ray $r$ is
\begin{equation} \label{eq:def-pd}
\Pr(b_{r}=i) = \binom{k}{i} q_{d}^{i} p_{d}^{k-i},\qquad 
p_{d} := 1-q_{d}.
\end{equation}
Consequently, we have
\begin{subequations} \label{eq:E-Xr}
\begin{align}
\Pr[X_{r}=1] &= \EE[X_{r}] = p_{d}^{k}, \\
\Pr[X_{r}=0] &= \sum_{i=1}^{k} \binom{k}{i} q_{d}^{i} p_{d}^{k-i}
= 1 - p_{d}^{k}.
\end{align}
\end{subequations}
\begin{lemma} \label{eq:lem-NR}
 The expected number of non-zero measurements defined by \eqref{eq:def-NR} is
\begin{equation} \label{eq:NR0-values}
\begin{aligned}
 N_{R} &= N_{R}(k) = |R| (1-p_{d}^{k}) 
 = D d^{D-1} \bigg(1-\Big(1-\frac{1}{d^{D-1}}\Big)^{k}\bigg), \\
 N_{R}^{0} &= N_{R}^{0}(k) = |R|-N_{R} = |R| p_{d}^{k}
 = D d^{D-1} \Big(1-\frac{1}{d^{D-1}}\Big)^{k}.
\end{aligned}
\end{equation}
\end{lemma}
\begin{proof}
Due to the linearity of expectation, summing over all rays gives
\[
N_{R} = \EE\Big[\sum_{r \in R} (1-X_{r}) \Big]
= |R| (1-p_{d}^{k}).
\]
\end{proof}
\begin{remark} \label{rem:m-red}
Note that $N_{R}$ specifies the expected value of $m_{red}$ in \eqref{def:mn-red} induced by random $k$-sparse vectors $x \in \R_{k,+}^{n}$. See Figure \ref{fig:NR-Plot} for an illustration.
\end{remark}
\begin{figure}
\centerline{
\includegraphics[width=0.5\textwidth]{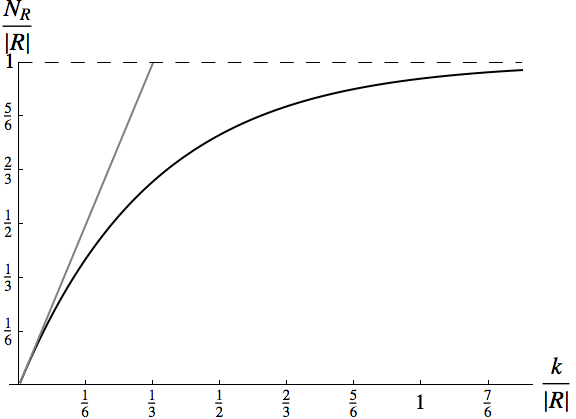}
}
\caption{
The expected number $N_{R}$ of non-zero measurements \eqref{eq:E-Xr}. For highly sparse scenarios (small $k$), the expected support \eqref{eq:def-NR} of the measurement vector $|\supp(b)| \approx 3 k$. For large values of $k$, this rate decreases due to the multiple incidence of cells and projection rays.
}
\label{fig:NR-Plot}
\end{figure}

%%%
\subsubsection*{Bounding the Deviation of $N_{R}^{0}$}
We are interested in how sharply the random number $X = \sum_{r \in R} X_{r}$ of zero measurements peaks around its expected value $N_{R}^{0} = \EE[X]$ given by \eqref{eq:NR0-values}. We derive next a corresponding tail bound by regarding a sequence of $k$ randomly located cells and by bounding the difference of subsequent conditional expected values of the random variable $X$. Theorem \ref{thm:Azuma} then provides a bound for the deviation $|X-\EE[X]|$.

Let the set of rays $R$ represent the elementary events corresponding to the observations $X_{r}=1$ or $X_{r}=0$ for each ray $r \in R$, i.e.~ray $r$ corresponds to a zero measurement or not. 

Let $\mc{F}_{i} \subset 2^{R},\, i=0,1,2,\dotsc$, denote the $\sigma$-field generated by the collection of subsets of $R$ that correspond to all possible events after having observed $i$ randomly selected cells. We set $\mc{F}_{0} = \{\emptyset,R\}$. Because observing cell $i+1$ just further partitions the current state based on the previously observed $i$ cells by possibly removing some ray (or rays) from the set of zero measurements, we have a nested sequence (filtration) $\mc{F}_{0} \subseteq \mc{F}_{1} \subseteq \dotsb \subseteq \mc{F}_{k}$ of the set $2^{R}$ of all subsets of $R$.

Based on this, for a fixed value of the sparsity parameter $k$, we define the sequence of random variables
\begin{equation}
 Y_{i} = \EE[X|\mc{F}_{i}],\quad i=0,1,\dotsc,k,
\end{equation}
where $Y_{i},\,i=0,1,\dotsc,k-1$, are the random variables specifying the expected number of zero measurements after having observed $k$ randomly selected cells, conditioned on the subset of events $\mc{F}_{i}$ determined by the observation of $i$ randomly selected cells. Consequently, $Y_{0}=\EE[X]=N_{R}^{0}$ due to the absence of any information, and $Y_{k} = X$ is just the observed number of zero measurements. 
The sequence $(Y_{i})_{i=0,\dotsc,k}$ is a martingale by construction satisfying $\EE[Y_{i+1}|\mc{F}_{i}]=Y_{i}$, that is condition \eqref{eq:condition-martingale}.
\begin{proposition} \label{prop:NR0-deviation}
 Let $N_{R}^{0}=\EE[X]$ be the expected number of zero measurements for a given sparsity parameter $k$, given by \eqref{eq:NR0-values}. Then, for any $\delta > 0$, 
\begin{equation} \label{eq:NR0-deviation}
\begin{aligned}
 \Pr\big(|X-N_{R}^{0}| \geq \delta\big) \;&\leq\;
 2 \exp\bigg(
 -\frac{1-p_{d}^{2}}{(1-p_{d}^{2k})}
 \;\frac{\delta^{2}}{2 D^{2}}
 \bigg) \\
 & \nearrow\;
 2 \exp\Big(-\frac{\delta^{2}}{2 D^{2} k}\Big) 
 \qquad\text{if}\quad d \to \infty.
\end{aligned}
\end{equation}
\end{proposition}
This result shows that for large problem sizes $d$ occurring in applications, concentration of observations of $N_{R}^{0}$ primarily depends on the sparsity parameter $k$. As a consequence, the bound enables suitable choices of $k = k(d)$ of the sparsity parameter depending on the problem size.

For example, typical values
\begin{equation} \label{eq:k-LaVision}
\begin{aligned}
k = \begin{cases} 
0.05 d & \text{in 2D}, \\
0.05 d^{2} & \text{in 3D},
\end{cases}
\end{aligned}
\end{equation}
chosen by engineers\footnote{Personal communication.} in applications according to a rule of thumb, result in
\begin{equation} \label{eq:deviation-concrete}
 \Pr\big(|X-N_{R}^{0}| \geq \delta\big) \;\leq\;
 \begin{cases}
 2 \exp\Big(-\frac{5}{2 d} \delta^{2} \Big) & \text{in 2D},\\
 2 \exp\Big(-\frac{10}{9 d^{2}} \delta^{2} \Big) & \text{in 3D}.
 \end{cases}
\end{equation}
For the 3D case \eqref{eq:k-LaVision}, the probability to observe deviations from $N_{R}^{0}$ larger than $1\%$ drops below $0.01$ for problem sizes $d \geq 77$, which is common in practice.

Thus, the bound \eqref{eq:NR0-deviation} is strong enough to indicate not only that \eqref{eq:k-LaVision} is a particular sensible choice, but also leads to more proper choices of $k$ for applications, which still give highly concentrated values of observations of $N_{R}^{0}$. This is the essential prerequisite for threshold effects of unique recovery from sparse measurements.
\begin{proof}[Proof (Proposition \ref{prop:NR0-deviation})]
 Let $R^{0}_{i-1} \subset R$ denote the subset of rays with zero measurements after the random selection of $i-1 < k$ cells. For the remaining $k-(i-1)$ trials, the probability that not any cell incident with some ray $r \in R^{0}_{i-1}$ will be selected, is 
\begin{equation}
 p_{d}^{k-(i-1)} = \EE[X_{r}|\mc{F}_{i-1}],
\end{equation}
with $p_{d}$ given by \eqref{eq:def-pd}. Consequently, by linearity, the expectation $Y_{i-1}$ of zero measurements given $|R^{0}_{i-1}|$ zero measurements after the selection of $i-1$ cells, is
\begin{equation}
 Y_{i-1} = \EE[X|\mc{F}_{i-1}]
 = \sum_{r \in R^{0}_{i-1}} p_{d}^{k-(i-1)}.
\end{equation}
Now suppose we observe the random selection of the $i$-th cell. We distinguish two possible cases.
\begin{enumerate}
\item
 Cell $i$ is not incident with any ray $r \in R^{0}_{i-1}$. Then the number of zero measurements remains the same, and
\begin{equation}
 Y_{i} = \sum_{r \in R^{0}_{i-1}} p_{d}^{k-i}.
\end{equation}
Furthermore,
\begin{equation} \label{eq:Azuma-estimate-1}
\begin{aligned}
Y_{i}-Y_{i-1} &= \sum_{r \in R^{0}_{i-1}} \big(
 p_{d}^{k-i} - p_{d}^{k-(i-1)} \big)
 = |R^{0}_{i-1}| p_{d}^{k-i} (1-p_{d}) \\
 &\leq (|R|-1) p_{d}^{k-i} q_{d}.
\end{aligned}
\end{equation}
\item
Cell $i$ is incident with $1, \dotsc,D$ rays contained in $R^{0}_{i-1}$. Let $R^{0}_{i}$ denote the set $R^{0}_{i-1}$ after removing these rays. Then
\[
 Y_{i} = \sum_{r \in R^{0}_{i}} p_{d}^{k-i}.
\]
Furthermore, since $R^{0}_{i} \subset R^{0}_{i-1}$ and
$|R^{0}_{i-1} \setminus R^{0}_{i}| \leq D$,
\begin{equation} \label{eq:Azuma-estimate-2}
\begin{aligned} 
Y_{i-1} - Y_{i} &=
\sum_{r \in R^{0}_{i-1} \setminus R^{0}_{i}} p_{d}^{k-(i-1)}
- \sum_{r \in R^{0}_{i}} \big(
p_{d}^{k-i} - p_{d}^{k-(i-1)} \big) \\
&\leq D p_{d}^{k-i+1} 
- \sum_{r \in R^{0}_{i}} 
p_{d}^{k-i} (1-p_{d}) \leq D p_{d}^{k-i+1}.
\end{aligned}
\end{equation}
\end{enumerate}
Comparing the bounds \eqref{eq:Azuma-estimate-1} and 
\eqref{eq:Azuma-estimate-2}, we have with $|R| q_{d} = D$,
\[
(|R|-1) q_{d} p_{d}^{k-i} = (D - q_{d}) p_{d}^{k-i}, \qquad\qquad
D p_{d} p_{d}^{k-i} = (D - D q_{d}) p_{d}^{k-i}.
\]
Thus, we take the larger bound \eqref{eq:Azuma-estimate-1}, drop the immaterial $-1$ in the first factor and compute
\[
\sum_{i=1}^{k} (D p_{d}^{(k-i)})^{2}
= D^{2} \frac{1 - p_{d}^{2 k}}{1-p_{d}^{2}}.
\]
Inserting $p_{d}$ from \eqref{eq:def-pd} and expanding in terms of $d^{-1}$ at $0$, we obtain
\begin{align*}
 \frac{1 - p_{d}^{2 k}}{1-p_{d}^{2}} = \begin{cases}
 k + (k-k^{2}) d^{-1} + \mc{O}(d^{-2}), & \text{in 2D} \\
 k + (k-k^{2}) d^{-2} + \mc{O}(d^{-4}), & \text{in 3D}
 \end{cases} \qquad\xrightarrow{d \to \infty} k.
\end{align*}
Applying Theorem \ref{thm:Azuma} completes the proof.
\end{proof}

\begin{figure}
\centerline{
\includegraphics[width=0.5\textwidth]{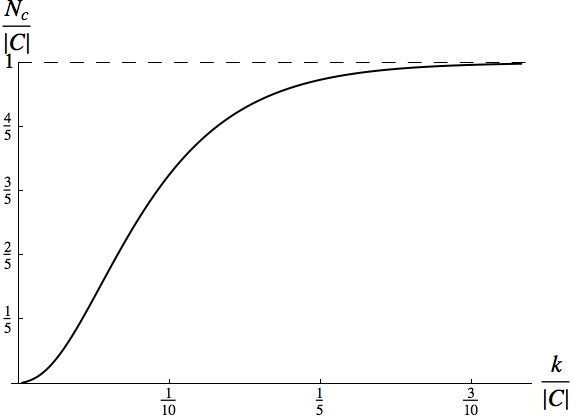}
}
\caption{
The expected number $N_{C} = \EE[|C_{b}|]$ of cells supporting observed measurement vectors $b$, given by \eqref{eq:NC-value}. Starting with rate $N_{C} \propto k$ for very small values of $k$, it quickly increases and exceeds $N_{R}$ (Fig.~\ref{fig:NR-Plot}), thus leading to underdetermined reduced systems \eqref{eq:red-system}.
}
\label{fig:NC-Plot}
\end{figure}

%%%
\subsubsection{Expected Number of Cells}

In the previous section, we computed the expected number of measurements $N_{R} = \EE[|\supp(b)|]$ induced by a random unknown $k$-sparse vector $x$ (Lemma \ref{eq:lem-NR}) along with a tail bound for $N_{R}^{0} = |R|-N_{R}$ (Prop.~\ref{prop:NR0-deviation}).

In the present section, we determine the expected number of cells corresponding to $N_{R}$, denoted by $N_{C}$. We confine ourselves to the practically more relevant 3D case.

As in the previous section, $X \in \{0,1\}^{|R|}$ denotes a random vector indicating subsets of projection rays. $X_{r}=1,\, r \in R$, corresponds to a zero observation along ray $r$. For a subset of rays $R_{b} \subset R$, we say that the corresponding subset of cells $C_{b}$ in \eqref{eq:def-RbCb} \textbf{supports} $R_{b}$.
\begin{proposition}\label{prop:NC}
 For a given value of the sparsity parameter $k$, the expected size of subsets of cells that support random subsets $R_{b} \subset R$ of observed non-zero measurements, is 
\begin{equation} \label{eq:NC-value}
 N_{C} = N_{C}(k) = d^{3} \bigg( 1 
 - 3\Big(1-\frac{1}{d^{2}}\Big)^{k} 
 + 3\Big(1-\frac{2 d-1}{d^{3}}\Big)^{k}
 - \Big(1-\frac{3 d-2}{d^{3}}\Big)^{k}
 \bigg).
\end{equation}
\end{proposition}
\begin{proof}
We partition the set of rays $R = R_{1} \cup R_{2} \cup R_{3}$ according to the three projection images (Fig.~\ref{fig_1}) and associate with the cells $C$ the corresponding set of triples of projection rays
\[
R_{1,2,3} = \big\{ (r_{1}, r_{2}, r_{3}) \colon \cap_{i=1}^{3} r_{i} \neq \emptyset,\; r_{i} \in R_{i},\; i=1,2,3 \big\},
\]
with each triple intersecting in a single cell. Thus, we have $|R_{1,2,3}| = |C| = d^{3}$, and each cell $c_{ijk} = r_{i} \cap r_{j} \cap r_{k}$ belongs to the set $C_{b}$ supporting $R_{b}$ if 
$R_{b} \cap (r_{i} \cup r_{j} \cup r_{k}) \neq \emptyset$. In terms of random variables $X_{r}$ indicating zero-measurements by $X_{r}=1$, this means that $c_{ijk} \in C_{b}$ if not $X_{r_{i}}=X_{r_{j}}=X_{r_{k}}=1$. Thus,
\begin{align*}
 N_{C} &= \EE\Big[
 \sum_{R_{1,2,3}} 
 (1-X_{r_{1}})(1-X_{r_{2}})(1-X_{r_{3}}) \Big] \\
 &=  \sum_{R_{1,2,3}} \Big( 1 - 
  \big(\EE[X_{r_{1}}] + \EE[X_{r_{2}}] + \EE[X_{r_{3}}]\big)
 + \sum_{1 \leq i < j \leq 3} \EE[X_{r_{i}} X_{r_{j}}]
 -  \EE[X_{r_{1}} X_{r_{2}} X_{r_{3}}] \Big). 
\end{align*}
This expression takes into account the intersection of projection rays $r_{i}, r_{j}$ (inclusion-exclusion principle) in order not to overcount the number of supporting cells.

We have $\EE[X_{r_{i}}] = p_{d}^{k} = (1-d^{-2})^{k}$ by \eqref{eq:E-Xr} and \eqref{eq:def-pd}.
The event $X_{r_{i}} X_{r_{j}}=1$ means that both rays correspond to zero measurements, which happens with probability
\[
\Big(1 - \frac{|r_{i} \cup r_{j}|}{|C|}\Big)^{k}
= \Big(1 - \frac{2 d-1}{d^{3}}\Big)^{k}.
\]
We have three pairs of sets of rays from $R = R_{1} \cup R_{2} \cup R_{3}$, and each of the $d^{2}$ rays $r_{i} \in R_{i}$ intersects with $d$ rays $r_{j} \in R_{j}$. Finally, three intersecting rays correspond to zero measurements with probability
\[
\Big(1 - \frac{|r_{1} \cup r_{2} \cup r_{3}|}{|C|}\Big)^{k}
= \Big(1 - \frac{3 d-2}{d^{3}}\Big)^{k},
\]
for each of the $d^{3}$ cells $c \in C$.
\end{proof}
\begin{remark} \label{rem:n-red}
Note that $N_{C}$ specifies the expected value of $n_{red}$ in \eqref{def:mn-red} induced by random $k$-sparse vectors $x \in \R_{k,+}^{n}$. See Figure \ref{fig:NC-Plot} for an illustration.
\end{remark}
%

%%%
\subsubsection{Overdetermined Reduced Systems: Critical Sparsity $k$}

For small value of $k$, that is for highly sparse scenarios, the expected value $N_{R}(k) \approx 3 k$ grows faster than $N_{C}(k) \approx k$. Consequently, the \emph{expected} reduced system due to Definition \ref{eq:red-system} will be overdetermined. This holds up to a critical value $k \leq k_{crit}$ because for increasing values of $k$, it is more likely that several particles are incident with some projection ray, making $N_{C}$ increasing faster than $N_{R}$.
\begin{proposition} \label{prop:kcrit}
 For $k \leq k_{crit}$, the reduced system \eqref{eq:red-system} will be overdetermined with high probability, where $k_{crit}$ solves
\begin{equation} \label{eq:def-k-crit}
 N_{R}(k_{crit}) = N_{C}(k_{crit})
\end{equation}
and $N_{R}(k_{crit}), N_{C}(k_{crit})$ are given by \eqref{eq:NR0-values} and \eqref{eq:NC-value}.
\end{proposition}
Figure \ref{fig:k-critical} shows the dependency $k_{crit} = k_{crit}(d)$ on the problem size $d$, as defined by \eqref{eq:def-k-crit} .
\begin{figure}
\centerline{
\includegraphics[width=0.5\textwidth]{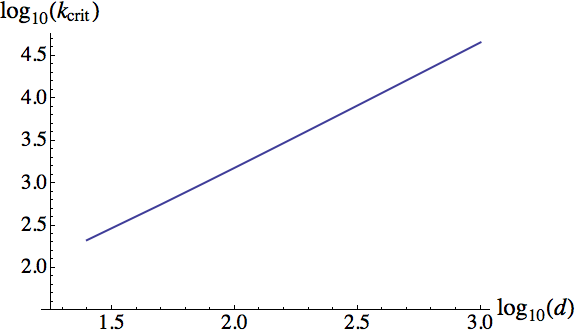}
}
\caption{Values $k_{crit} = k_{crit}(d)$ of the sparsity parameter such that $k \leq k_{crit}$ yield overdetermined reduced systems \eqref{eq:red-system}. For the depicted and practically relevant range of $d$, the slope of the $\log$-$\log$ curve slightly decreases in $1.65 \dots 1.55$.}
\label{fig:k-critical}
\end{figure}

%%%%%%%%%
\subsection{Unperturbed Systems}
\label{sec:recovery-unperturbed}

We consider the recovery properties of the 3D setup depicted in Fig.~\ref{fig_1}, based on Theorem \ref{thm:SP} and on the \emph{expected} quantities involved in the corresponding condition \eqref{eq:condition-Wang}, as worked out in Section \ref{sec:reduced-system}. Concerning the interpretation of the following claims, we refer to Remark \ref{rem:probability}.
\begin{proposition} \label{prop:appl-Wang}
 The system $A x = b$, with measurement matrix $A$ given by \eqref{eq:def-AdD}, admits unique recovery of $k$-sparse non-negative vectors $x$ with high probability, if
\begin{subequations}
\begin{gather} \label{eq:k-unpertubed}
 k \leq \frac{N_{C}(k_{\delta})}{1+\delta}
 = \frac{1}{3 \delta (1+\delta)} N_{R}(k_{\delta}),\qquad
 \delta > \frac{\sqrt{5}-1}{2},
 \intertext{where $k_{\delta}$ solves}
 N_{R}(k_{\delta}) = 3 \delta N_{C}(k_{\delta})
 \label{eq:def-kdelta}
\end{gather}
\end{subequations}
and $N_{R}(k), N_{C}(k)$ are given by \eqref{eq:NR0-values} and \eqref{eq:NC-value}.
\end{proposition}
\begin{proof}
 The assertion follows from replacing the quantities forming condition \eqref{eq:condition-Wang} by their expected values, due to Remarks \ref{rem:m-red} and \ref{rem:n-red}.
\end{proof}
\begin{remark} \label{rem:unperturbed}
Equation \eqref{eq:def-kdelta} shows that unique recovery of a $k$-sparse, $k\le\frac{n_{red}}{(1+\delta)}$, non-negative vector can be expected using the unperturbed measurement matrix provided the reduced system \eqref{eq:red-system} is by a factor $m_{red} \geq 1.854 \, n_{red}$ overdetermined. See Figure \ref{fig:k-All} for an illustration.
\end{remark}

%%%%%%%%%
\subsection{Perturbed Systems}
\label{sec:recovery-perturbed}

Analogously to the previous section, we evaluate the average recovery performance using perturbed systems based on Theorem \ref{thm:SPCS1}.
\begin{proposition} \label{prop:appl-Hassibi}
 The system $\tilde A x = b$, with perturbed measurement matrix $\tilde A$ given by \eqref{eq:def-AdD}, admits unique recovery of $k$-sparse non-negative vectors $x$ with high probability, if $k$ satisfies condition $k \leq k_{crit}$ from Prop.~\ref{prop:kcrit}, that is, if the reduced system \eqref{eq:red-system} is overdetermined.
\end{proposition}
\begin{proof}
 Immediate from Theorem \ref{thm:SPCS1}, replacing the quantities forming condition \eqref{eq:Hassibi-condition} by their expected values, and taking into account $\ell=3$ for the measurement matrix \eqref{eq:def-AdD} and the case $D=3$.
\end{proof}
\begin{remark}
 In view of this assertion and Remark \ref{rem:unperturbed}, it is remarkable that a significant gain of recovery performance can be obtained by a simple device: structure-preserving perturbation of the measurement matrix. See Figure \ref{fig:k-All} for an illustration.
\end{remark}

\begin{figure}
\centerline{
\includegraphics[width=0.6\textwidth]{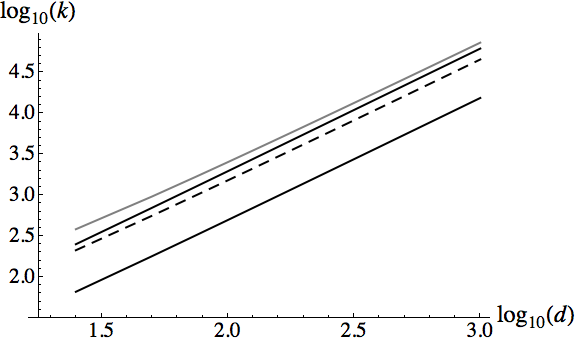}
}
\caption{
Critical upper bound sparsity values $k = k(d)$ that guarantee unique recovery of $k$-sparse vectors $x$ on average with high probability. From bottom to top: $k_{\delta}$ \eqref{eq:k-unpertubed} 
for unperturbed matrices $A$, $k_{crit}$ \eqref{eq:def-k-crit} resulting in overdetermined reduced systems, $k_{max}$ \eqref{eq:kmax-criterion} for underdetermined perturbed matrices $A$, and fully random measurement matrices.
}
\label{fig:k-All}
\end{figure}

%%%%%%%%%
\subsection{Underdetermined Perturbed Systems}
\label{sec:recovery-perturbed-underdet}

Based on \eqref{eq:WendelPR} and the average case analysis of condition \eqref{eq:Hassibi-condition} (Section 
\ref{sec:reduced-system}), we devise a criterion for determining the maximal sparsity value $k$ (minimal sparse scenario), such that any $k$-sparse vector $x$ can be uniquely recovered with high probability using the measurement matrix $A$ given by \eqref{eq:def-AdD}.
Unlike Propositions \ref{prop:appl-Wang} and \ref{prop:appl-Hassibi}, we specifically consider here less sparse scenarios that result in \emph{underdetermined} reduced systems \eqref{eq:red-system}. 
\begin{proposition} \label{prop:kmax}
Let $A$ be a matrix satisfying the assumptions of Lemma \ref{lem:Hassibi-4-2} with $\tilde r_{0} = N_{R}(k_{max})$, where $k_{max}$ solves
\begin{equation} \label{eq:kmax-definition}
 N_{R}(\tilde k_{max}) = \delta N_{C}(\tilde k_{max}),\qquad
 \delta > \frac{\sqrt{5}-1}{2},
\end{equation}
with $N_{R}(k), N_{C}(k)$ given by \eqref{eq:NR0-values} and \eqref{eq:NC-value}. Then a $k$-sparse vector $x$ can be uniquely recovered with high probability, if
\begin{equation} \label{eq:kmax-criterion}
k \leq k_{max} = \frac{N_{R}(\tilde k_{max})}{3}.
\end{equation}
\end{proposition}
\begin{proof}
By assumption and Lemma \ref{lem:Hassibi-4-2}, Theorem \ref{thm:Hassibi-4-1} (see also Remark \ref{rem:Hassibi-recovery}) implies \eqref{eq:kmax-criterion}, thereby taking into account that Eqn.~\eqref{eq:kmax-definition} defining $k_{max}$ reflects the expected version of condition \eqref{eq:condition-Wang}, subdivided by the factor $3$ due to \eqref{eq:kmax-criterion}.
\end{proof}

\cite{Wendel-62,CoverSeparability-65,Man09ProbInteger}
Figure \ref{fig:k-All} illustrates the value $k_{max}$ \eqref{eq:kmax-criterion} and compares it to the previous results.

\vspace{0.5cm}
Finally, we comment on the uniqueness condition established in \cite{Man09ProbInteger} which corresponds to the top $k(d)$ curve in Figure \ref{fig:k-All}. This result does not apply to our setting.
The reason is that a basic assumption underlying the application of \eqref{eq:WendelPR} does \emph{not} hold. While after some perturbation the points corresponding to the columns of $\tilde A$ and the sparsity value $|I^{-}(x)|=k$ are in general position, the underlying distribution lacks symmetry with respect to the origin. As a result, we cannot establish the superior performance of ``fully'' random sensors considered in \cite{Man09ProbInteger}. 

\subsection{Two Cameras are Not Enough}
\label{sec:2-cam}
In the present section, we briefly discuss how the previously obtained bounds on sparsity apply in the 2D scenario. To this end, we first compute the expected value of nonempty cells connected to $R_{b}$ measurements
generated by a $k$ sparse nonnegative vector.
\begin{proposition}\label{prop:NC-2D}
In 2D,  the expected size of subsets of cells that support random subsets $R_{b} \subset R$ of observed non-zero measurements, is 
\begin{equation} \label{eq:NC-value-2}
 N_{C} = N_{C}(k) = d^{2} \bigg( 1 - \Big(1-\frac{1}{d}\Big)^{k}\bigg)^2\ ,
\end{equation}
for a given sparsity parameter $k$,
\end{proposition}
\begin{proof}
We partition the set of rays $R = R_{1} \cup R_{2}$ according to the two projection images (Fig.~\ref{fig_1}), left, and associate with the cells $C$ the corresponding set of pairs of projection rays
\[
R_{1,2} = \big\{ (r_{1}, r_{2}) \colon \cap_{i=1}^{2} r_{i} \neq \emptyset,\; r_{i} \in R_{i},\; i=1,2 \big\},
\]
with each pair intersecting in a single cell. Thus, we have $|R_{1,2}| = |C| = d^{2}$, and each cell $c_{ij} = r_{i} \cap r_{j} $ belongs to the set $C_{b}$ supporting $R_{b}$ if 
$R_{b} \cap (r_{i} \cup r_{j}) \neq \emptyset$. In terms of random variables $X_{r}$ indicating zero-measurements by $X_{r}=1$, this means that $c_{ij} \in C_{b}$ if not $X_{r_{i}}=X_{r_{j}}=1$. Thus,
\begin{align*}
 N_{C} &= \EE\Big[
 \sum_{R_{1,2}} 
 (1-X_{r_{1}})(1-X_{r_{2}}) \Big] \\
 &=  \sum_{R_{1,2}} \Big( 1 - 
  \big(\EE[X_{r_{1}}] + \EE[X_{r_{2}}] \big)
 + \sum_{1 \leq i < j \leq 2} \EE[X_{r_{i}} X_{r_{j}}]
] \Big), 
\end{align*}
taking  the intersection of projection rays $r_{i}, r_{j}$ into account.
We obtained $\EE[X_{r_{i}}] = p_{d}^{k} = (1-\frac{1}{d})^{k}$ in \eqref{eq:E-Xr} and \eqref{eq:def-pd}.
The event that both rays correspond to zero measurements $X_{r_{i}} X_{r_{j}}=1$ happens with probability
\[
\Big(1 - \frac{|r_{i} \cup r_{j}|}{|C|}\Big)^{k}
= \Big(1 - \frac{2 d-1}{d^{2}}\Big)^{k} = \Big(1 - \frac{1}{d}\Big)^{2k}.
\]
\end{proof}

By Prop. \ref{prop:NC-2D} and Lemma \ref{eq:lem-NR} we can now compute the
the expected ratio of the dimensions of the reduced system, further denoted by $c$.
We solve the polynomial $N_R(k)=c N_C(k)$ according to and \eqref{eq:NC-value-2}.
Interesting are the values $c\in\{2\delta,1,\delta,\frac{1}{2}\}$.
For example, if $c=2\delta$, we obtain guaranteed recovery of all $1$-sparse vectors,
which also equals the strong threshold for the 2D case.
If $c=1$, we obtain, on average, that any $k$-sparse vector $x$, with
\begin{equation}
  \label{eq:Er-critical}
  k\le k_{crit}=\frac{\log\left( \frac{d-2}{d}\right)}{\log\left( \frac{d-1}{d}\right)}\approx 2\ ,
\end{equation}
induces reduced reduced overdetermined systems. Thus two particles can always be reconstructed,
after perturbation.
If $c=\frac{1}{2}$ the critical sparsity value approximately equals $4$ for arbitrary $d$. This is the best
achievable bound, which is obviously useless for application. For $k=3$ it can be shown that
the probability of correct recovery via the perturbed matrix $A_d^2$ is
\begin{equation*}
1-\frac{2\cdot 4\cdot \binom{d}{2}\binom{d}{3}+4\cdot \binom{d}{3}^2}{\binom{d^3}{3}}=
\frac{ d^2 + 6 d -10 }{3 (d^2-2)}\xrightarrow{d \to \infty} 1/3\ .
\end{equation*}
We mention that the expected relative values of $N_R$ and $N_C$ do not vary much with
different two camera arrangements. This highly pessimistic results can be explained by the fact that 
there is no expander with constant left degree $\ell$ less than $3$.

%%%%%%%%%%%  
\section{Numerical Experiments and Discussion}
\label{sec:experiments}

In this section we empirically investigate bounds on the required sparsity that
guarantee unique nonnegative or binary $k$-sparse solutions.

\subsection{Reduced Systems versus Analytical Sparsity Thresholds}
\label{sec:frac}
The workhorse of the previous theoretical average case performance analysis of the discrete
tomography matrix from \eqref{eq:A_dD}
is the derivation of the expected number of nonzero rows $N_R(k)$ induced by the
$k$-sparse vector along with the number $N_C(k)$ of ''active'' cells which cannot be empty.
This can be done also empirically, see Fig. \ref{fig:Frac_2D_3D}, left, for the 2D case
and right, for the 3D case. To generate the figures we varied $k\in\{1,2,\cdots, 2000\}$
and $d\in\{10,11,\cdots, 100\}$ in 2D and $k\in\{1,2,\cdots, 2000\}$
and $d\in\{10,11,\cdots, 100\}$ in 3D, respectively, and generated for each point $(k,d)$
500 problem instances. The plots show $N_R(k,d)/N_C(k,d)$
along with the curves: $k_{\delta}$ \eqref{eq:k-unpertubed} for unperturbed matrices $A$, $k_{crit}$ 
\eqref{eq:def-k-crit} resulting in overdetermined reduced systems, $k_{max}$ \eqref{eq:kmax-criterion} for underdetermined perturbed matrices $A$, and $k_{opt}$ which solves
\begin{equation}\label{eq:k_opt} 
N_R(k_{opt})=0.5 N_C(k_{opt})\ .
\end{equation} 
\begin{figure}
\begin{center}
\begin{tabular}{cc}
\includegraphics[clip,width=0.45\textwidth]{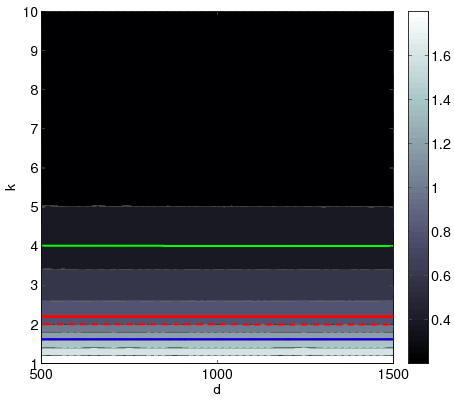} &
\includegraphics[clip,width=0.45\textwidth]{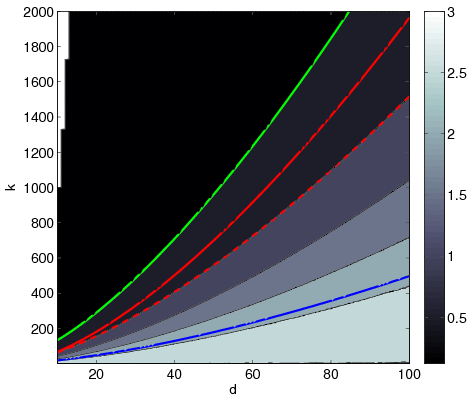}\\
\end{tabular}
\end{center}
\caption{The contourplots of the average fraction of the reduced systems as a function
of the resolution parameter $d$ and the sparsity parameter $k$.
{\bf Left:} In agreement with the results of Section 5.5, the plots for the 2D case  show that level lines of $N_R(k,d)/N_C(k,d)$ are constant with varying $d$. 
{\bf Right:} In 3D the situation dramatically changes. Higher sparsity values are allowed for increasing
values of $d$, as the derived threshold curves show.
Below the blue curve $k_{\delta}$ \eqref{eq:k-unpertubed} reconstruction for unperturbed systems is 
guaranteed with high probability.
Below the dashed red curve $k_{crit}$ \eqref{eq:def-k-crit} reduced systems are overdetermined. For points below the solid red curve  $k_{max}$ \eqref{eq:kmax-criterion} reconstruction is guaranteed for perturbed systems. Finally, problem instances under the green curve $k_{opt}$ \eqref{eq:k_opt} could be recovered if the reduced matrices would follow a symmetrical distribution with respect to the origin.}
\label{fig:Frac_2D_3D}
\end{figure}

\subsection{Empirical Phase Transitions}
We further concentrate on the 3D case. In analogy to \cite{DonTan05} we assess the so called
\emph{phase transition} $\rho$ as a function of $d$, which
is reciprocally proportional to the undersampling ratio $\frac{m}{n}\in(0,1)$.
We consider $d\in\{10,11,\dots,100\}$, the corresponding matrix $A^3_d\in\R^{3d^2\times d^3}$ 
from \eqref{eq:A_dD} and its perturbed version $\tilde A$ and the sparsity as a fraction of $d^2$,
$k=\rho d^2$, for $\rho\in (0,1)$. 

This phase transition $\rho(d)$ indicates the necessary relative sparsity
to recover a $k$-sparse solution with overwhelming probability. More precisely, 
if $\|x\|_0\le\rho(d) \cdot d^2$, then with
overwhelming probability a random $k$-sparse nonnegative (or binary) vector $x^*$ is the unique
solution in $\calF_+:=\{x \colon Ax= Ax^*, x\ge 0\}$ or 
$\calF_{ 0,1}:=\{x \colon Ax= Ax^*, x\in[0,1]^n\}$, 
respectively. 
Uniqueness can be ''verified'' by minimizing and maximizing the same objective $f^\top x$
over $\calF_+$ or $\calF_{ 0,1}$, respectively. If the minimizers coincide for
several random vectors $f$ we claim uniqueness.
As shown in Fig. \ref{fig:sliceA3D} the threshold for a unique nonnegative solution
and a unique  $0/1$-bounded solution are quite close.

To generate the success and failure transition plots 
we generated $A$ according to \eqref{eq:A_dD} and $\tilde A$ by slightly perturbing 
its entries and varying $d\in\{10,11,\dots,100\}$ 
$\tilde A$ has the same sparsity structure as $A$, but random entries 
drawn from the standard uniform distribution on the open interval 
$(0.9,1.1)$. We have tried different perturbation levels, all leading
to similar results. Thus we adopted this interval for all
presented results.

Then for $\rho\in[0, 1]$ a $\rho d^2$-sparse  nonnegative or binary vector
was generated to compute the right hand side measurement vector
and for each $(d,\rho)$-point 50 random problem instances
were generated. A threshold-effect is clearly visible in all figures exhibiting 
parameter regions where the probability of exact reconstruction is
close to one and it is much stronger for the perturbed systems. The results are in excellent
agreement with the derived analytical thresholds. We refer to the 
figure captions for detailed explanations. Finally, we refer to
the summary in Figure \ref{fig:concl} for the computed sharp sparsity thresholds, 
which  are in excellent agreement with our numerical experiments.
%%%
\begin{figure}[h]
\begin{center}
\includegraphics[clip,width=0.55\textwidth]{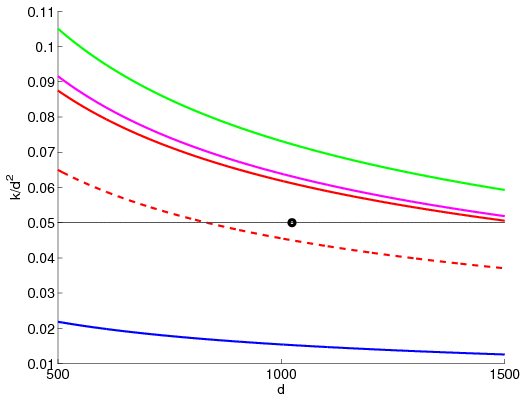}\\
\end{center}
\caption{ Relative critical upper bound sparsity values $k(d)$ in the practical relevant domain $d\in(500,1500)$ that guarantee unique recovery of $k$-sparse vectors $x$ on average with high probability. From bottom to top: $k_{\delta}$ \eqref{eq:k-unpertubed} for unperturbed matrices $A$ (blue line), $k_{crit}$ \eqref{eq:def-k-crit} resulting in overdetermined reduced systems (dashed red line), $k_{max}$ \eqref{eq:kmax-criterion} and $\tilde k_{max}$ \eqref{eq:kmax-definition} for underdetermined perturbed matrices $A$ (solid red and pink line),  and ideal random measurement matrices $k_{opt}$ (green line). The thin black line depicts the particle density
used by engineers in practice, while the black spot corresponds to the typical resolution parameter
$d=1024$. The results demonstrate that specific slight random perturbations of the TomoPIV measurement matrix considerably boost the expected reconstruction performance by at least 150\%.}
\label{fig:concl}
\end{figure}
%%%%

\begin{figure}
\begin{center}
\begin{tabular}{cc}
\includegraphics[clip,width=0.38\textwidth]{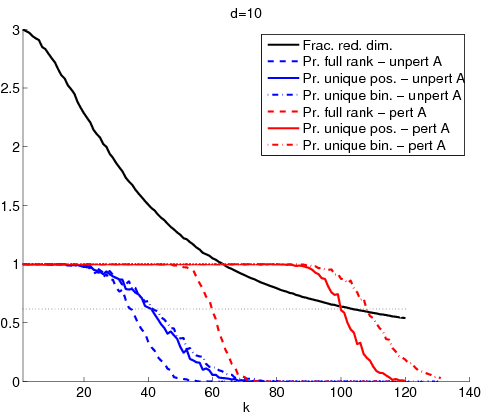} &
\includegraphics[clip,width=0.3\textwidth]{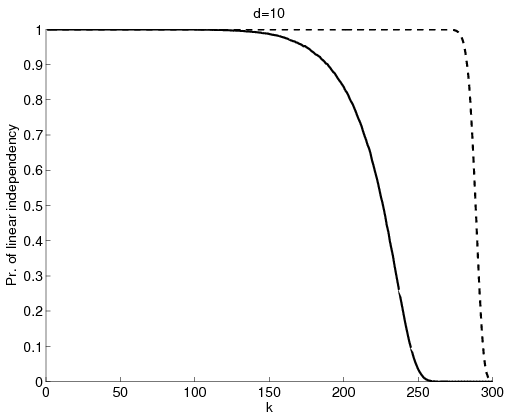}\\
\includegraphics[clip,width=0.38\textwidth]{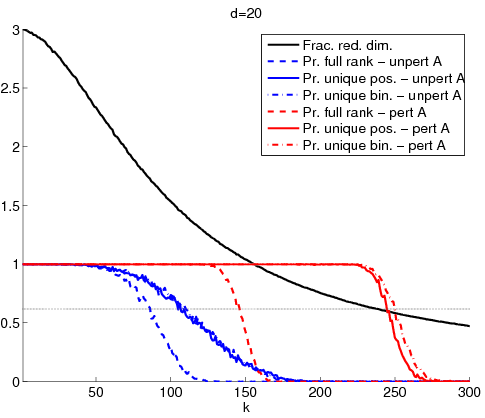} &
\includegraphics[clip,width=0.3\textwidth]{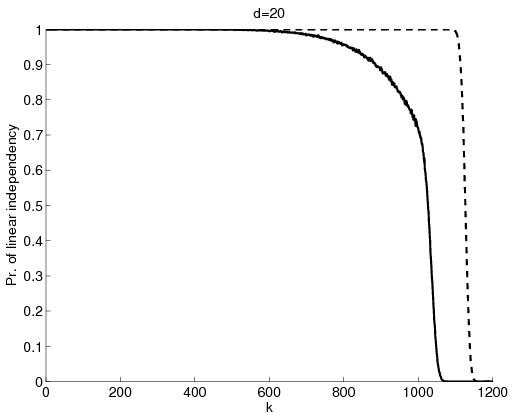}\\
\includegraphics[clip,width=0.38\textwidth]{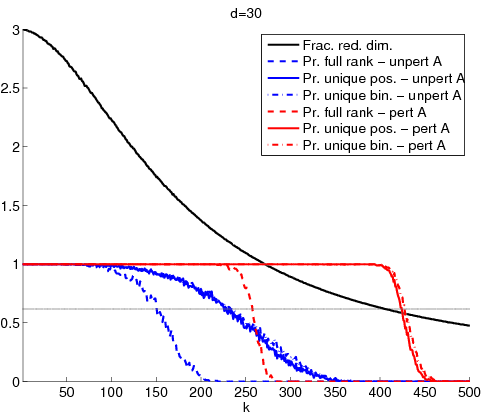}&
\includegraphics[clip,width=0.3\textwidth]{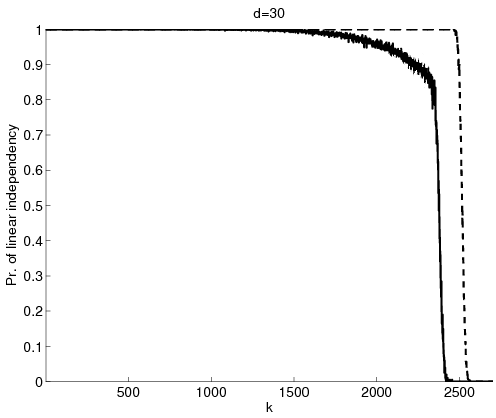}\\
\end{tabular}
\end{center}
\caption{
{\bf Left:} Recovery via the unperturbed matrix $A^3_d$ (blue curves), $d\in\{10,20,30\}$ (from top to down) 
versus the perturbed counterpart (red curves).
The dash-dot line depicts the empirical probability (500 trials) that reduced systems are overdetermined and have full rank. The solid line (blue: unperturbed, red: perturbed) shows the probability that a $k$-sparse nonnegative vector is unique. The dashed curve  shows the probability that a $k$-sparse binary solution is
the  unique solution of in $[0,1]^n$. Additional information like binarity gives only a slight performance
boost. The curve $k_{\delta}$ \eqref{eq:k-unpertubed} correctly predicts that
18 ($d=10$), 48 ($d=20$), and 85 ($d=30$) particle are
reconstructed with high probability via the unperturbed systems  and 
66 ($d=10$), 181 ($d=20$), 328 ($d=30$) particles, via the perturbed systems according to
$k_{max}$ \eqref{eq:kmax-criterion}. 
However, 105 ($d=10$),  241 ($d=20$),  408  ($d=30$), by $\tilde k_{max}$ from 
\eqref{eq:kmax-definition} are more accurate. Division by three does not seem to be necessary.
{\bf Right:} Empirical probability obtained from 10000 trials that $k$ random columns
of the unperturbed matrix (solid black line) or of the perturbed matrix (dashed black line)
are linearly independent.}
\label{fig:sliceA3D}
\end{figure}
\begin{figure}
\begin{center}
\begin{tabular}{cc}
\includegraphics[clip,width=0.43\textwidth]{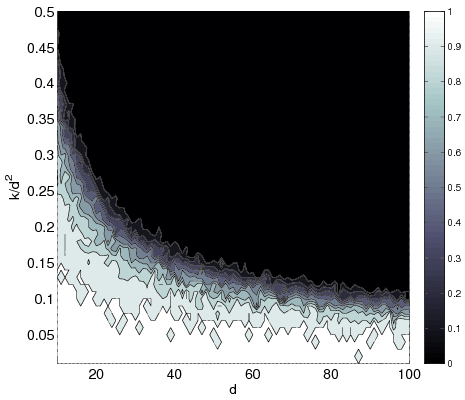} &
\includegraphics[clip,width=0.43\textwidth]{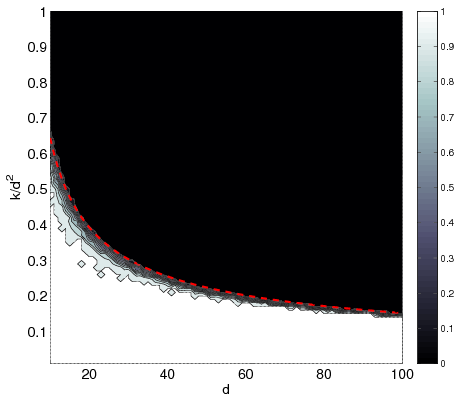}\\
\includegraphics[clip,width=0.43\textwidth]{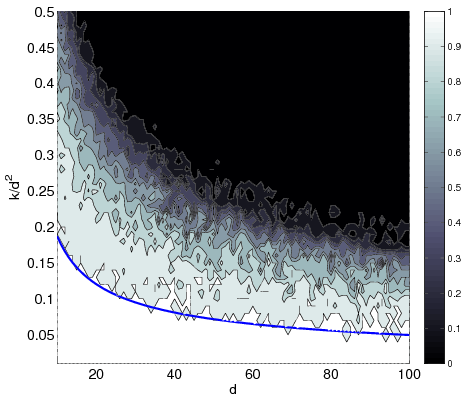} &
\includegraphics[clip,width=0.43\textwidth]{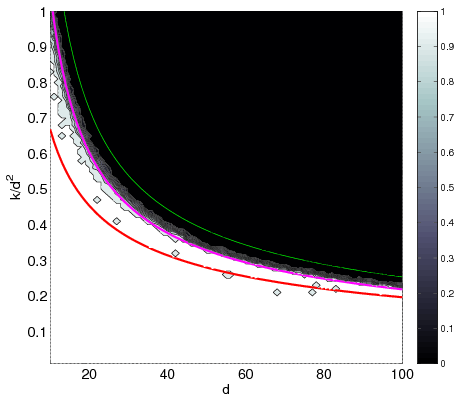}\\
\includegraphics[clip,width=0.43\textwidth]{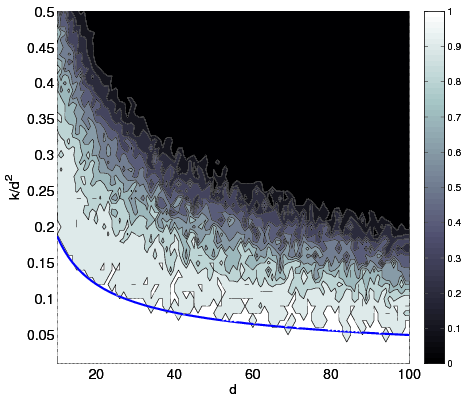} &
\includegraphics[clip,width=0.43\textwidth]{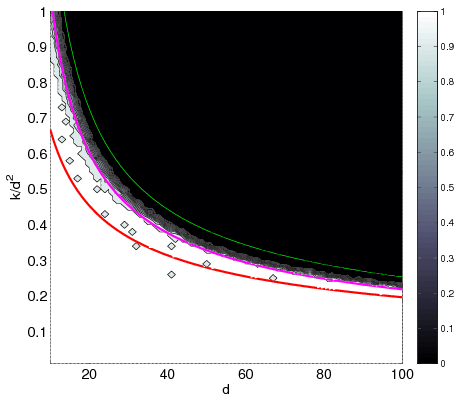}\\
\end{tabular}
\end{center}
\caption{{\bf Left:} Success and failure empirical phase transitions for unperturbed and perturbed systems
{\bf right}. {\bf Top:} Probability that the reduced matrices are overdetermined
and of full rank, along ({\bf right}) with the estimated relative critical sparsity level $k_{krit}$ 
(dashed red line) which induces overdetermined reduced matrices.
{\bf Middle:} Probability of uniqueness of a $k=\rho d^2$ sparse nonnegative vector.
{\bf Bottom:} Probability of uniqueness in $[0,1]^n$ of a $k=\rho d^2$ sparse binary vector.
The blue curve depicts again $k_{\delta}$ \eqref{eq:k-unpertubed}, the dashed red curve $k_{crit}$ 
\eqref{eq:def-k-crit}, the solid red curve  $k_{max}$ \eqref{eq:kmax-criterion}, $\tilde k_{max}$ 
\eqref{eq:kmax-criterion}
and the green curve $k_{opt}$ \eqref{eq:k_opt}.
In case of the perturbed matrix $\tilde A$ exact recovery is possible \emph{beyond} 
overdetermined reduced matrices. Moreover $\tilde k_{max}$ follows most accurately the empirical phase
transition for perturbed systems.
}
\label{fig:contour3D}
\end{figure}

\section{Conclusions}
The main contribution of this work is the transfer of recent results
on compressive sensing via expander graphs with bad expansion properties
to the discrete tomography problem. In particular, we consider a sparse
binary measurement matrix, which encodes the incidence relation between
projection rays and image discretization cells, along with its
slightly perturbed counterpart.  While the expected expansion of the underlying graph
does not change with perturbation, the recovery performance can be boosted significantly. 
We investigate the average performance in recovery of exact sparse nonnegative signals by 
analyzing the properties of reduced systems obtained by eliminating zero measurements and related redundant discretization cells. We compute sharp sparsity thresholds, 
such that the maximal sparsity can be determined precisely for both perturbed and 
unperturbed scenarios. Our theoretical analysis suggests that a
similar procedure can be applied to different geometries.
%%%

\bibliographystyle{plain}
%\bibliography{LAA-Tomo}

\end{document}